\documentclass[12pt, reqno]{amsart}

\pdfoutput=1

\usepackage[english]{babel}

\usepackage[letterpaper,top=2cm,bottom=2cm,left=3cm,right=3cm,marginparwidth=1.75cm]{geometry}

\usepackage{amsmath, amsthm}
\usepackage{graphicx}
\usepackage[colorlinks=true, allcolors=blue]{hyperref}

\usepackage{amsmath, amssymb, amsthm, fullpage, color, comment, mathtools, tikz, dsfont, tikz-cd}

\usepackage[
backend=biber,
style=alphabetic,
maxbibnames=99,
doi=false,
isbn=false,
url=false,
minalphanames=3,
maxalphanames=3
]{biblatex}

\usepackage{adjustbox}

\theoremstyle{plain}
\newtheorem{theorem}{Theorem}[section]
\newtheorem{proposition}[theorem]{Proposition}
\newtheorem{lemma}[theorem]{Lemma}
\newtheorem{corollary}[theorem]{Corollary}

\theoremstyle{definition}

\theoremstyle{remark}
\newtheorem{remark}[theorem]{Remark}
\newtheorem{example}[theorem]{Example}

\title{The discrete wave equation with applications to scattering theory and quantum chaos}

\author{Carsten Peterson}

\begin{document}
    \begin{abstract}
        With a view towards studying the multitemporal wave equation on affine buildings recently introduced by Anker--Rémy--Trojan \cite{anker_remy_trojan}, we systematically develop the basic properties of the discrete wave equation on $\mathbb{Z}$ and use this to explain existing results about the wave equation on regular graphs. Furthermore, we explicitly compute the incoming and outgoing translation representations and the scattering operator, in the sense of Lax--Phillips, for regular and biregular trees. Finally, we use the wave equation on biregular graphs to extend a result of Brooks--Lindenstrauss about delocalization of eigenfunctions on regular graphs to the setting of biregular graphs.
    \end{abstract}

    \maketitle

        \tableofcontents

\section{Introduction}

The wave equation (say, on $\mathbb{R}^n$) defined by $\frac{\partial^2}{\partial t^2} u(x, t) = \Delta_x u(x, t)$, is one of the most classical and well-studied partial differential equations in mathematics and physics. It affords many remarkable properties such as existence and uniqueness of solutions given Cauchy data, finite speed of propagation/Huygens' principle, d'Alembert's solution (and more generally solutions expressible as superpositions of plane waves), and conservation of energy. From an algebraic perspective, many of these geometric and analytic properties may be viewed as a consequence of the fact that the wave equation involves applying ``the same'' (at least when $n = 1$) differential operator in space and in time, and that this operator is the simplest non-trivial operator invariant under $x \mapsto -x$. This perspective shall play a distinguished role in the sequel.

In this paper we begin by studying the \textit{discrete wave equation} on $\mathbb{Z}$, i.e. functions $u: \mathbb{Z} \times \mathbb{Z} \to \mathbb{C}$ satisfying
\begin{gather}
    u(n, t+1) + u(n, t-1) = u(n+1, t) + u(n-1, t).  \label{eqn_discrete_wave_intro}
\end{gather}
If one were to subtract the term $2 u(n, t)$ from both sides of \eqref{eqn_discrete_wave_intro}, one would immediately recognize that \eqref{eqn_discrete_wave_intro} simply involves setting the discrete second derivative in the second (i.e. ``time'') variable equal to the discrete second derivative in the first (i.e. ``space'') variable. We shall at times refer to this as the \textit{flat wave equation} to contrast it with the wave equation on trees and graphs to be considered later on. 

While \eqref{eqn_discrete_wave_intro} is certainly not new, we have been unable to find a systematic treatment in the literature. Note that, if we were to convert a function on $\mathbb{Z} \times \mathbb{Z}$ into a formal bi-infinite Laurent series in $x$ and $y$, then the LHS of \eqref{eqn_discrete_wave_intro} amounts to multiplication by $y + y^{-1}$, and the RHS amounts to multiplication by $x + x^{-1}$. As we shall show, most of the exceptional properties of \eqref{eqn_discrete_wave_intro} are a result of algebraic properties of the Laurent polynomial ring $\mathbb{C}[x, x^{-1}]$ (in fact, many of the remarkable properties of the wave equation on $\mathbb{R}^n$ are a result of algebraic properties of $\mathbb{C}[x]$; see Section \ref{sec_wave_eqn_R}). Let $W = S_2$ be the symmetric group on two elements whose non-trivial element $w_0$ acts on $\mathbb{C}[x, x^{-1}]$ via $x \mapsto x^{-1}$. Then $\mathbb{C}[x, x^{-1}]$ is a free module of rank 2 over $\mathbb{C}[x, x^{-1}]^W = \mathbb{C}[x + x^{-1}]$ (Proposition \ref{prop_alg_properties}). Choosing appropriate initial conditions to guarantee existence and uniqueness of solutions ultimately amounts to choosing a basis for this module. A very natural, and in many ways best possible, basis is given by $h_1 = 1, h_2 = \frac{x - x^{-1}}{2}$. Furthermore, $\mathbb{C}[x, x^{-1}]$ carries a natural non-degenerate $\mathbb{C}[x + x^{-1}]$-bilinear pairing $(\cdot, \cdot)$ taking values in $\mathbb{C}[x + x^{-1}]$ (Proposition \ref{prop_pairing}). This allows us to find a dual basis for any basis; the dual basis to $h_1, h_2$ is $m_1 = \frac{x - x^{-1}}{2}, m_2 = 1$. The relationship between initial conditions and the corresponding fundamental solutions is essentially the relationship between a basis and the corresponding dual basis (Theorem \ref{thm_soln}). The d'Alembert presentation of solutions to the wave equation is essentially a consequence of the definition of $(\cdot, \cdot)$ (Corollary \ref{cor_d'alembert}), finite speed of propagation is immediate from the explicit expression of the solutions (Corollary \ref{cor_finite_speed}), and conservation of energy amounts to diagonalizing the matrix $[(m_j h_j, x)]$ (Proposition \ref{prop_energy_flat}). Furthermore, as we discuss in Section \ref{sec_chebyshev}, studying ``fundamental solutions'' to \eqref{eqn_discrete_wave_intro} essentially amounts to studying Chebyshev polynomials, or more generally Bernstein-Szegö polynomials. Orthogonality properties of such polynomials can also be understood using this framework. See Theorem \ref{thm_orthogonality}.

When considering the wave equation on the hyperbolic plane $\mathbb{H}$ (or on hyperbolic surfaces), often one studies the \textit{shifted wave equation} (also referred to as the \textit{automorphic wave equation} in the context of hyperbolic surfaces) defined as solutions $u: \mathbb{H} \times \mathbb{R} \to \mathbb{C}$ to the equation
\begin{gather}
    \frac{\partial^2}{\partial t^2} u(x, t) = \Big(\Delta_x - \frac{1}{4}\Big) u(x, t). \label{shifted_wave_eqn}
\end{gather} 
See, for instance, \cite{pavlov_feddeev, lax_phillips_automorphic, lax_phillips_lattice_points, anker_pierfelice_vallarino, borthwick}. The appearance of $\Delta-\frac{1}{4}$ on the RHS may be explained as follows. There is a natural algebra isomorphism $\Gamma$ (a special case of the Harish-Chandra isomorphism) from $\textnormal{SL}(2, \mathbb{R})$-invariant differential operators on $\mathbb{H}$ to polynomials on $\mathbb{C}$ invariant under $x \mapsto -x$; the latter is simply $\mathbb{C}[x^2]$. We have that $\Gamma(\Delta - \frac{1}{4}) = x^2$, which is the symbol of the operator $\frac{\partial^2}{\partial x^2}$ which appears on the RHS of the wave equation on $\mathbb{R}$. The wave equation \eqref{shifted_wave_eqn} has nicer properties than the usual wave equation on $\mathbb{H}$, partially because it interfaces much better with the representation theory of $\textnormal{SL}(2, \mathbb{R})$. One may ultimately study \eqref{shifted_wave_eqn} by a two-step process: first by studying the wave equation on $\mathbb{R}$, and then studying the nature of the map $\Gamma$ (the latter of which can be studied via spherical harmonic analysis on $\mathbb{H}$).

Let $\mathcal{T}$ be the infinite $(q+1)$-regular tree. The discrete analogue of \eqref{shifted_wave_eqn} is
\begin{gather}
    u(v, t+1) + u(v, t-1) = \Big(\frac{A}{\sqrt{q}} \Big) u(v, t) \label{tree_wave_intro}
\end{gather}
where $u: \mathcal{T} \times \mathbb{Z} \to \mathbb{C}$, and $A$ is the adjacency operator on $\mathcal{T}$ (we shall only consider functions on the vertices of $\mathcal{T}$). Note that \eqref{tree_wave_intro} was independently and roughly simultaneously introduced by \cite{brooks_lindenstrauss} as the \textit{$p$-adic wave equation} (see also \cite{brooks_lindenstrauss_graphs, brooks_le_masson_lindenstrauss}), and by \cite{anker_martinot_pedon_setti} as the \textit{shifted wave equation}. In fact, we also found the same equation considered in the much earlier work \cite{romanov_rudin_95}. There is a natural map $\textnormal{Sat}$ (a special case of the Satake isomorphism) from $\textnormal{Aut}(\mathcal{T})$-invariant operators on $\mathcal{T}$ (also known as the spherical Hecke algebra) to $\mathbb{C}[x, x^{-1}]^W$. We have that $\textnormal{Sat}(\frac{A}{\sqrt{q}}) = x + x^{-1}$, which is essentially the ``symbol'' of the operator appearing on the RHS of \eqref{tree_wave_intro}. After discussing the relevant tools from spherical harmonic analysis on $\mathcal{T}$ in Section \ref{sec_harmonic_analysis_trees}, we quickly and systematically recover in Section \ref{sec_wave_eqn_tree} the main results of \cite{brooks_lindenstrauss} and \cite{anker_martinot_pedon_setti} concerning \eqref{tree_wave_intro}. We remark that many ideas closely related to the wave equation on the regular tree have implicitly appeared for a long time in analytic number theory in terms of Chebyshev polynomials of the normalized Hecke operator; see for instance \cite{serre_hecke, conrey_duke_farmer, kaplan_petrow, sarnak_zubrilina}. In case $q$ is a power of a prime, we may view $\mathcal{T}$ as the Bruhat--Tits tree associated to $\textnormal{SL}(2, F)$, with $F$ a non-archimedean local field with residue field of order $q$, and we may view the adjacency operator $A$ as the Hecke operator, which plays an important role in $p$-adic representation theory and the theory of automorphic forms.

Lax--Phillips established a general framework for studying wave-like equations, often referred to as the \textit{Lax--Phillips scattering theory}. Already in the book of Lax--Phillips \cite{lax_phillips_book}, the basics of the discrete analogue of this theory are established. The existence of a conserved energy for the wave equation \eqref{tree_wave_intro} (Proposition \ref{prop_tree_energy}), implies that the wave propagator acts unitarily on the Hilbert space of finite energy initial conditions. Lax--Phillips define the abstract notion of an \textit{outgoing subspace} $D_+$ in the context of a unitary operator $\mathcal{V}$ acting on a Hilbert space $\mathcal{H}$; intuitively this corresponds to initial data for which the corresponding solution to the wave equation is ``outgoing'', i.e. at time $t \geq 0$ it is supported outside a ball of radius $t$. Under such a set-up, Lax--Phillips prove the abstract existence of a so-called \textit{outgoing translation representation} $T_+: \mathcal{H} \to \ell^2(\mathbb{Z}; N)$, where $N$ is some auxiliary Hilbert space, such that $D_+$ maps bijectively to $\ell^2(\mathbb{Z}_{\geq 0}; N)$ and $\mathcal{V}$ turns into the right shift operator. The map $T_+$ is unique up to an isomorphism of $N$. There is an analogous notion of \textit{incoming translation representation} $T_-$ for \textit{incoming subspaces}. The relationship between these two maps is represented by the \textit{scattering operator}. Physically, waves ``incoming from infinity'' evolve over time in a potentially complicated way, but eventually ``scatter'' into a superposition of waves ``outgoing to infinity.'' The scattering operator encodes this information. In \cite{lax_phillips_h3}, Lax--Phillips describe very explicitly the nature of the translation representations and scattering operator for the shifted wave equation on hyperbolic 3-space. In Section \ref{sec_scattering_theory}, we obtain analogous results to many of the main results of \cite{lax_phillips_h3} but in the context of the wave equation on the regular tree. Along the way, we show how to express every solution to the wave equation as a superposition of horocyclic plane waves, and prove to what extent such a representation is unique.

The original applications of Brooks--Lindenstrauss' introduction of \eqref{tree_wave_intro} were in the field of \textit{quantum chaos}. In particular in \cite{brooks_lindenstrauss_graphs}, they show that any eigenfunction of the adjacency operator on a finite regular graph cannot concentrate its mass on too small of a set; the precise statement involves constants depending on certain geometric features of the underlying graph (such as the number of short cycles). Because the underlying ``classical dynamics'' on the graph (e.g. the random walk) is chaotic in the sense that it rapidly equidistributes, we heuristically expect that stationary quantum particles on the graph, represented by eigenfunctions of the adjacency operator, should also be fairly equidistributed, i.e. not localized to small sets. In Section \ref{sec_wave_eqn_biregular}, we introduce and study the wave equation on the biregular tree. Its definition is less obvious than in the case of the regular tree, but from a structural standpoint it is clear that it is the ``right'' equation to study. Using properties of this wave equation, we are able to adapt the technique of Brooks--Lindenstrauss to prove delocalization results for eigenfunctions on biregular graphs (Theorem \ref{thm_delocalization}). We also remark that in \cite{brooks_lindenstrauss}, Brooks--Lindenstrauss use \eqref{tree_wave_intro} to ultimately show quantum unique ergodicity for joint eigenfunctions of the Laplacian and one single Hecke operator on arithmetic hyperbolic surfaces. Here the interpretation of the $(p+1)$-regular tree as the Bruhat--Tits tree of $\textnormal{SL}(2, \mathbb{Q}_p)$ plays a crucial role. Seeing as the $(r, r^3)$-biregular tree is the Bruhat--Tits tree of $\textnormal{SU}(3, E/F)$ of an unramified quandratic extension $E$ of $F$, with $F$ having residue field of order $r$, we hope that ultimately the wave equation on biregular tree may have applications to arithmetic quantum unique ergodicity on locally symmetric spaces associated to the real Lie group $\textnormal{SU}(2, 1)$.

While many parts of this paper are expository in the sense that they discuss existing results, we believe that our approach to the discrete wave equation clarifies and simplifies many of these previous results. Furthermore, it opens the door to a similar analysis in higher rank settings. The wave equation on regular trees has been recently generalized to arbitrary regular affine buildings (of reduced type) in Anker--Rémy--Trojan \cite{anker_remy_trojan} (note that biregular trees are example of regular affine buildings of non-reduced type). In joint work in preparation with these authors \cite{multitemporal_buildings}, we plan to generalize essentially all of the results of this paper to the higher rank setting. This amounts in some sense to a two-step process. Let $\mathcal{B}$ be an affine building. Let $P$ denote the coweight lattice of the underlying root system, and let $W$ denote the Weyl group. We first study a ``flat'' multitemporal wave equation on $P \times P$:
\begin{gather}
    u(\lambda, \mu) *_{\mu} g(\mu) = u(\lambda, \mu) *_{\lambda} g^\vee(\lambda), \label{intro_higher_rank_wave}
\end{gather}
for every $g(\nu) \in \mathbb{C}[P]^W$, and where $g^\vee(\nu) := g(-\nu)$. Properties of this wave equation reduce to algebraic properties of the algebra $\mathbb{C}[P]$ (when the root system is of type $A_1$, i.e. associated to $\textnormal{SL}(2)$, then \eqref{intro_higher_rank_wave} is simply \eqref{eqn_discrete_wave_intro} and $\mathbb{C}[P] = \mathbb{C}[x, x^{-1}]$). The Satake isomorphism is an algebra isomorphism between the algebra of spherical averaging operators on special vertices (often referred to as the spherical Hecke algebra in the Bruhat--Tits setting), and $\mathbb{C}[P]^W$. Let $\mathcal{B}_s$ denote the set of special vertices of $\mathcal{B}$. The solutions to the multitemporal wave equation on $\mathcal{B}$ corresponds to functions $u: \mathcal{B}_s \times P \to \mathbb{C}$ satisfying
\begin{gather}
    u(v, \mu) *_{\mu} g(\mu) = \textnormal{Sat}^{-1}(g^\vee)_v u(v, \mu) \label{intro_higher_rank_building}
\end{gather}
for all $g \in \mathbb{C}[P]^W$. Note that \eqref{intro_higher_rank_building} reduces to \eqref{tree_wave_intro} in case the root system is of type $A_1$. We also remark that \eqref{intro_higher_rank_building} is modelled off of the multitemporal wave equation on symmetric spaces \cite{semenov, shahshahani_83, shahshahani_89, phillips_shahshahani, helgason_multitemporal} which, in the case of $\mathbb{H}$, reduces to \eqref{shifted_wave_eqn}. In \cite{multitemporal_buildings}, we shall show that properties of \eqref{intro_higher_rank_building} can be deduced via properties of \eqref{intro_higher_rank_wave} together with tools of spherical harmonic analysis on buildings.

\subsection*{Acknowledgements}
I would like to thank Jean-Philippe Anker, Bertrand Rémy, and Bartosz Trojan for innumerable helpful conversations related to this project. The project has received funding from
the European Union’s Horizon 2020 research and
innovation programme under the Marie Skłodowska-Curie
grant agreement No 101034255. This work was partially supported by U.S. National Science Foundation Grant DMS-2503324.

\section{The discrete wave equation on $\mathbb{Z}$}
We summarize the contents of this section. We remark that this section is mostly expository in the sense that most of the results of this section were already known, though we have not been able to find an analogous treatment to ours in the literature. We reiterate that part of the goal of this section (and this paper as a whole) is to set up a framework for studying the wave equation which readily generalizes to the higher rank setting.

In Section \ref{sec_notation}, we establish some basic notation that will be used throughout. In Section \ref{sec_algebra}, we discuss some algebraic properties of the ring $\mathbb{C}[x, x^{-1}]$. In Section \ref{sec_flat_soln}, we use these properties to prove existence and uniqueness of solutions of the flat wave equation. In Section \ref{sec_d'alembert}, we deduce from this result finite speed of propagation and the d'Alembert solution to the flat wave equation. In Section \ref{sec_chebyshev} we discuss the connection of the flat wave equation with the theory of Chebyshev polynomials, and more generally with Bernstein-Szegö polynomials. In Section \ref{sec_energy}, we derive the invariant energy form for the discrete wave equation. Finally, in Section \ref{sec_wave_eqn_R}, we revisit the wave equation on $\mathbb{R}$ and derive its basic properties in an analogous way to how we derive them in the discrete case.

\subsection{Notation} \label{sec_notation}
Suppose $h: \mathbb{Z} \to \mathbb{C}$. Instead of viewing $h$ as a function on $\mathbb{Z}$, we may instead represent it as a bi-infinite Laurent series in $x^{\pm 1}$. We denote this by $h^{[x]}$:
\begin{gather*}
    h^{[x]} := \sum_{n \in \mathbb{Z}} h(n) x^n.
\end{gather*}
We may clearly multiply $h^{[x]}$ by any Laurent polynomial. In the reverse direction, if $j(x)$ is a bi-infinite Laurent series in $x^{\pm 1}$, we use the notation $[x^n]j(x)$ to denote the coefficient of $x^n$ in $j(x)$.

Note that $h^{[x]}$ is in some sense the formal discrete Fourier-Laplace transform of $h(n)$, and $[x^n]j(x)$ reads off the $n$th Fourier coefficient of $j(x)$, assuming we can analytically make sense of $j(x)$ as a function on $S^1$. Very often though we shall treat these as algebraic operations. When we instead think of the underlying objects analytically, we shall use the variable $\xi$ in place of $x$ to emphasize the analytic, rather than algebraic, nature of the function in a specific context.

We shall extend the above notation to functions $u: \mathbb{Z} \times \mathbb{Z} \to \mathbb{C}$, i.e. we define $u^{[x, y]}$ as the formal bi-infinite Laurent series in $x^{\pm 1}, y^{\pm 1}$:
\begin{gather*}
    u^{[x, y]} := \sum_{(n, t) \in \mathbb{Z}^2} u(n, t) x^n y^t.
\end{gather*}
Again, we may multiply $u^{[x, y]}$ by any Laurent polynomial. With respect to this notation, we may rewrite \eqref{eqn_discrete_wave_intro} as
\begin{gather}
    u^{[x, y]} \cdot (x + x^{-1} - y - y^{-1}) = 0. \label{eqn_wave_laurent}
\end{gather}
We also make use of the notation
\begin{gather*}
    u^{[y]}(n, \cdot) := \sum_{t \in \mathbb{Z}} u(n, t) y^t,
\end{gather*}
which we may think of as a $\mathbb{Z}$-parametrized family of bi-infinite Laurent series. We may similarly define $u^{[x]}(\cdot, t)$. Notice that, if $u$ solves the flat wave equation, then
\begin{gather*}
    u^{[y]}(\cdot, t+1) + u^{[y]}(\cdot, t-1) = (y + y^{-1}) u^{[y]}(\cdot, t).
\end{gather*}
Any function $w: \Omega \times \mathbb{Z} \to \mathbb{C}$, where $\Omega \subseteq \mathbb{C}^\times$, which satisfies
\begin{gather}
    w(\xi, t+1) + w(\xi, t-1) = (\xi + \xi^{-1}) w (\xi, t) \label{spectral_wave_eqn}
\end{gather}
is said to satisfy the \textit{spectral wave equation}.

If $g(x, y)$ is a formal bi-infinite Laurent series, we define $[x^n y^t]g(x, y)$ as the coefficient of $x^n y^t$ in $g(x, y)$. We shall also use the notation
\begin{gather*}
    [y^t]g(x, y) := \sum_{n \in \mathbb{Z}} [x^n y^t]g(x, y) \cdot x^n.
\end{gather*}

\subsection{Algebraic properties of $\mathbb{C}[x, x^{-1}]$ and $\mathbb{C}[x, x^{-1}]^W$} \label{sec_algebra}
Let $\mathbb{C}[\mathbb{Z}]$ denote the group algebra of $\mathbb{Z}$. This is clearly isomorphic to $\mathbb{C}[x, x^{-1}]$, the ring of Laurent polynomials. Let $W = S_2 = \{1, w_0\}$ denote the symmetric group on 2 elements. The ring $\mathbb{C}[x, x^{-1}]$ carries a $W$-action by having the non-trivial element $w_0$ act by $x \mapsto x^{-1}$. From a more representation theoretic perspective, we should think of $\mathbb{Z}$ as the coweight lattice of $\mathfrak{sl}(2)$, $W$ as the Weyl group, and the above action as being induced from the natural action of $W$ on the coweight lattice. 

It is a basic fact that $\mathbb{C}[x, x^{-1}]^W = \mathbb{C}[x + x^{-1}]$, i.e. every $W$-invariant Laurent polynomial is a polynomial in $x + x^{-1}$; we shall call such polynomials \textit{symmetric}. We say that an element in $\mathbb{C}[x, x^{-1}]$ is \textit{antisymmetric} if it gets multiplied by $-1$ upon applying $w_0$. Another basic fact is that every antisymmetric polynomial is a product of $x - x^{-1}$ with a symmetric polynomial. Another way of expressing this is that for any $f \in \mathbb{C}[x, x^{-1}]$, we have $\frac{f - w_0.f}{x - x^{-1}} \in \mathbb{C}[x, x^{-1}]^W$. The following additional properties are well-known, but we prove them here for completeness.

\begin{proposition} \label{prop_alg_properties}
    We have the following basic properties of $\mathbb{C}[x, x^{-1}]$.
    \begin{enumerate}
        \item $\mathbb{C}[x, x^{-1}]$ is a rank 2 free module over $\mathbb{C}[x, x^{-1}]^W$.
        \item We can take $1, \frac{x - x^{-1}}{2}$ as a free basis.
        \item As a $W$-module, the $\mathbb{C}$-vector space $\mathbb{C}[x, x^{-1}]/(x + x^{-1})$ is isomorphic to the regular representation.
    \end{enumerate}
\end{proposition}

\begin{proof}
    Claims (1) and (2) above say that every element $f \in \mathbb{C}[x, x^{-1}]$ can be written uniquely as $f = g_1 + g_2 \frac{(x - x^{-1})}{2}$ with $g_1, g_2 \in \mathbb{C}[x + x^{-1}]$. Granting this for now, if we write $g_i = a_i + b_i$ with $a_i \in \mathbb{C}$ and $b_i \in (x + x^{-1})$ (the ideal in $\mathbb{C}[x + x^{-1}]$ generated by $x + x^{-1}$), we get that $f \equiv a_1 + a_2 \frac{(x - x^{-1})}{2}$ mod $(x + x^{-1})$. This in particular implies that $1, \frac{x - x^{-1}}{2}$ span $\mathbb{C}[x, x^{-1}]/(x + x^{-1})$ as a vector space. Furthermore, they must clearly be linearly independent since for example $1$ gets fixed by the $W$-action, and $\frac{x - x^{-1}}{2}$ transforms according to the sign character under the $W$-action. This shows that $\mathbb{C}[x, x^{-1}]/(x + x^{-1})$ is indeed isomorphic to the regular representation for $W$.

    We are thus reduced to showing Claims (1) and (2). We adapt an argument of Steinberg \cite{steinberg} who applied it in much greater generality. Consider the following matrix and its inverse
    \begin{gather*}
        A = \begin{bmatrix}
            1 & \frac{x - x^{-1}}{2} \\
            1 & - \frac{x - x^{-1}}{2}
        \end{bmatrix}, \hspace{5mm} A^{-1} = \begin{bmatrix}
            \frac{1}{2} & \frac{1}{2} \\
            \frac{1}{x - x^{-1}} & - \frac{1}{x - x^{-1}}
        \end{bmatrix}.
    \end{gather*}
    Suppose $f \in \mathbb{C}[x, x^{-1}]$. Inside of $\mathbb{C}(x)$, the equation
    \begin{gather}\label{eqn_permute_soln}
        A \begin{bmatrix}
            a_1 \\
            a_{w_0}
        \end{bmatrix} = \begin{bmatrix}
            f \\
            w_0.f
        \end{bmatrix}
    \end{gather}
    has the unique solution
    \begin{gather*} 
    \begin{bmatrix} 
        a_1 \\
        a_{w_0}
    \end{bmatrix} = A^{-1} \begin{bmatrix}
        f \\
        w_0.f
    \end{bmatrix} = \begin{bmatrix}
        \frac{1}{2} (f + w_0.f) \\
        \frac{f - w_0.f}{x - x^{-1}}
    \end{bmatrix}.
    \end{gather*}
    Note that we actually have $\frac{f - w_0.f}{x - x^{-1}} \in \mathbb{C}[x, x^{-1}]^W$.

    We can write 
    \begin{gather*}
        f = \frac{1}{2}(f + w_0.f) \cdot 1 + \frac{f - w_0.f}{x - x^{-1}} \cdot \frac{x - x^{-1}}{2}.
    \end{gather*}
    Because the second row of $A$ is simply the element $w_0$ applied to each entry in the first row, we get that any time $a_1, a_{w_0}$ are such that $a_1 + a_{w_0} \frac{x - x^{-1}}{2} = f$, we must necessarily have that they solve \eqref{eqn_permute_soln}. By the uniqueness of the solution to \eqref{eqn_permute_soln} we conclude uniqueness of $a_1, a_{w_0}$. 
\end{proof}

Consider the following symmetric pairing on $\mathbb{C}[x, x^{-1}]$ taking values in $\mathbb{C}[x, x^{-1}]^W$:
\begin{gather} \label{eqn_pairing}
    (f, g) := \frac{fg - w_0 (f g)}{x - x^{-1}}.
\end{gather}
The following is essentially due to Kazhdan-Lusztig \cite{kazhdan_lusztig_87}. For completeness, we include the argument here.

\begin{proposition} \label{prop_pairing}
    Elements $g_1, g_2 \in \mathbb{C}[x, x^{-1}]$ form a free basis for $\mathbb{C}[x, x^{-1}]$ as a module over $\mathbb{C}[x, x^{-1}]^W$ if and only if the Gram matrix with respect to the pairing in \eqref{eqn_pairing} has determinant which is a non-zero constant.
\end{proposition}

\begin{proof}
    Suppose $g_1, g_2 \in \mathbb{C}[x, x^{-1}]$. By Proposition \ref{prop_alg_properties} we know that there exists a matrix $B$ with entries in $\mathbb{C}[x + x^{-1}]$ such that 
    \begin{gather*}
        B \begin{bmatrix}
            1 \\
            \frac{x - x^{-1}}{2}
        \end{bmatrix} = \begin{bmatrix}
            g_1 \\
            g_2
        \end{bmatrix}.
    \end{gather*}
    Thus $g_1, g_2$ is a free basis if and only if $B^{-1}$ also has entries in $\mathbb{C}[x + x^{-1}]$. This occurs if and only if $\det(B)$ is a non-zero constant as these are the only units in $\mathbb{C}[x + x^{-1}]$.

    Let $f_1 = 1$, $f_2 = \frac{x - x^{-1}}{2}$. We have
    \begin{gather} \label{eqn_expand_matrix}
        (x - x^{-1}) \begin{bmatrix}
            (f_1, f_1) & (f_1, f_2) \\
            (f_2, f_1) & (f_2, f_2)
        \end{bmatrix} = \begin{bmatrix}
            f_1 & w_0.f_1 \\
            f_2 & w_0.f_2
        \end{bmatrix} \begin{bmatrix}
            1 & \\
             & -1
        \end{bmatrix} \begin{bmatrix}
            f_1 & f_2 \\
            w_0.f_1 & w_0.f_2
        \end{bmatrix}.
    \end{gather}
    The third matrix on the RHS is the matrix $A$ from the proof of Proposition \ref{prop_alg_properties}, and the first matrix on the RHS is $A^T$. They both have determinant $-(x - x^{-1})$. Therefore, the RHS has determinant $-(x - x^{-1})^2$. On the other hand the LHS has determinant $(x - x^{-1})^2$ times the determinant of the Gram matrix for $f_1, f_2$; we conclude that the determinant of this Gram matrix is therefore $-1$.

    Likewise, the matrix whose entries are $(x - x^{-1}) (g_i, g_j)$ is obtained by multiplying the RHS of \eqref{eqn_expand_matrix} on the right by $B^T$ and on the left by $B$. Thus the Gram matrix for the basis $g_1, g_2$ has determinant $-\det(B)^2$ which is a non-zero constant if and only if $\det(B)$ is as well.
\end{proof}

Given a free basis $h_1, h_2$ of $\mathbb{C}[x, x^{-1}]$ we get a unique dual basis $m_1, m_2$ with respect to the pairing in \eqref{eqn_pairing}. Concretely, the dual basis is found by inverting the Gram matrix and applying it to $[h_1, h_2]^T$. The dual basis of $1, \frac{x - x^{-1}}{2}$ is $\frac{x - x^{-1}}{2}, 1$. On the other hand the dual basis of $1, x$ is $-x^{-1}, 1$. 

We have the following useful relationship.

\begin{proposition} \label{prop_basis_dual}
    Suppose $c \in \mathbb{C}[x, x^{-1}]$. Let $h_1, h_2 \in \mathbb{C}[x, x^{-1}]$ be a free basis of $\mathbb{C}[x, x^{-1}]$ as a $\mathbb{C}[x, x^{-1}]^W$-module, and let $m_1, m_2$ be the dual basis. Then
    \begin{gather*}
        c = (m_1, c) h_1 + (m_2, c) h_2.
    \end{gather*}
\end{proposition}
\begin{proof}
    This can be seen immediately by applying the pairing with $m_1$ and $m_2$ to both sides.
\end{proof}
As we shall see in the next section, the relationship between a basis and its dual basis is essentially the same thing as the relationship between initial conditions and fundamental solutions to the wave equation.

\subsection{Initial conditions and fundamental solutions} \label{sec_flat_soln}
Let $h \in \mathbb{C}[x, x^{-1}]$ with $h = \sum_i c_i x^i$ (with $c_i = 0$ for all but finitely many $i$). We may think of $h$ as defining an ``initial condition'' for functions $u : \mathbb{Z} \times \mathbb{Z} \to \mathbb{C}$. More precisely, we say that $u(n, t)$ has $h$-initial condition equal to $g: \mathbb{Z} \to \mathbb{C}$ if $\sum_i c_i u(n, i) = g(n)$ for all $n$. In terms of the notation previously introduced, we may re-express this as:
\begin{gather*}
    [y^0](u^{[y]}(n, \cdot) \cdot h(y^{-1}) ) = g(n).
\end{gather*}

\begin{theorem} \label{thm_soln}
    Suppose $h_1, h_2$ is a free basis for $\mathbb{C}[x, x^{-1}]$ as a module over $\mathbb{C}[x, x^{-1}]^W$. Let $m_1, m_2$ be the dual basis. Then given any two functions $g_1, g_2 : \mathbb{Z} \to \mathbb{C}$, the unique solution $u : \mathbb{Z}^2 \to \mathbb{C}$ satisfying the wave equation with initial conditions
    \begin{align}
        [y^0] \Big( u^{[y]}(n, \cdot) \cdot h_1(y^{-1}) \Big) &= g_1(n) \label{eqn_flat_initial} \\
        [y^0] \Big( u^{[y]}(n, \cdot) \cdot h_2(y^{-1}) \Big) &= g_2(n) \notag
    \end{align}
    is given by:
    \begin{gather} \label{eqn_soln_initial}
        u(n, t) = [x^n] \Big( g_1^{[x]} \cdot (m_1, x^t) + g_2^{[x]} \cdot (m_2, x^t) \Big).
    \end{gather}
\end{theorem}
\begin{proof}
    First we show uniqueness. A function $u : \mathbb{Z}^2 \to \mathbb{C}$ is determined by knowing $[x^n y^0](u^{[x, y]} \cdot y^{-t})$ (i.e. $u(n, t)$) for all $n, t \in \mathbb{Z}$. We can write $y^t = a_1^{(t)} h_1(y) + a_2^{(t)} h_2(y)$ for appropriate $a_i^{(t)} \in \mathbb{C}[y + y^{-1}]$; in fact we know by Proposition \ref{prop_basis_dual} that $a_i^{(t)} = (m_i(y), y^t)$. Applying $w_0$, we get $y^{-t} = a_1^{(t)} h_1(y^{-1}) + a_2^{(t)} h_2(y^{-1})$. 

    Suppose $u(n, t)$ solves the wave equation. Suppose we know that $[x^n y^0](u^{[x, y]} \cdot h_i(y^{-1})) = g_i(n)$. Then
    \begin{align*}
        u(n, t) &= [x^n y^0] (u^{[x, y]} \cdot y^{-t}) \\
        &= [x^n y^0] \Big(u^{[x, y]} \cdot \big( a_1^{(t)}(y) \cdot h_1(y^{-1}) + a_2^{(t)}(y) \cdot h_2(y^{-1}) \big) \Big) \\
        &= [x^n y^0] \Big(u^{[x, y]} \cdot \big( a_1^{(t)}(x) \cdot h_1(y^{-1}) + a_2^{(t)}(x) \cdot h_2(y^{-1}) \big) \Big) \\
        &= [x^n] \big(g_1^{[x]} \cdot a_1^{(t)} (x) + g_2^{[x]} \cdot a_2^{(t)}(x) \big). 
    \end{align*}
    To go from the second line to the third, we have used the fact that if $u(n, t)$ solves the wave equation, then for every $a \in \mathbb{C}[x + x^{-1}]$, we have that $u^{[x, y]} \cdot a(y) = u^{[x, y]} \cdot a(x)$. To go from the third line to the fourth, we use the initial conditions, together with the following simple observation: suppose $b(x, y)$ is a formal Laurent series, and let $c(x)$ be the Laurent series in $x$ such that $[x^n] c(x) = [x^n y^0] b(x, y)$. Then for any Laurent polynomial $d(x)$ we have $[x^n y^0]\big (b(x, y) \cdot d(x) \big) = [x^n] \big(c(x) \cdot d(x) \big)$. We apply this observation to $b(x, y) = u^{[x, y]} \cdot h_i(y^{-1})$, in which case $c(x) = g_i^{[x]}$. We thus see that knowing $[x^n y^0]\big (u^{[x, y]} \cdot h_i(y^{-1}) \big)$ is enough to know the value of $u(n, t)$ for all $n, t$ assuming $u$ solves the wave equation. We thus get uniqueness of solution.

    We now show existence. The previous paragraph shows that if a solution to the wave equation with the prescribed initial conditions exists, it must be of the form in \eqref{eqn_soln_initial}. We are thus just tasked with showing that \eqref{eqn_soln_initial} has the desired properties.

    First we show that the expression in \eqref{eqn_soln_initial} solves the wave equation. Let $u_i(n, t) = [x^n] \big(g_i^{[x]} \cdot (m_i, x^t) \big)$. Then 
    \begin{align*}
        u_i(n+1, t) + u_i(n-1, t) &= [x^n] \big( g_i^{[x]} \cdot (m_i, x^t) \cdot (x + x^{-1}) \big) \\
        u_i(n, t+1) + u_i(n, t-1) &= [x^n] \Big( g_i^{[x]} \cdot \big( (m_i, x^{t+1}) + (m_i, x^{t-1}) \big) \Big).
    \end{align*}
    These are clearly equal, so $u(n, t)$ solves the wave equation.

    Lastly, we check the initial conditions. Suppose, for example $h_1 = \sum_i c_i x^i$. Then
    \begin{align*}
        [y^0] (u_1^{[y]}(n, \cdot) \cdot h_1(y^{-1})) &= \sum_{i} c_i \cdot u_1 (n, i) = \sum_{i} c_i \cdot [x^n]\big(g_1^{[x]} \cdot (m_1, x^i) \big) \\
        &= [x^n] \big(g_1^{[x]} \cdot (m_1, h_1) \big) = g_1(n).
    \end{align*}
    Similarly we get $[y^0](u_1^{[y]}(n, \cdot) \cdot h_2(y^{-1})) = 0$. Arguing similarly for $u_2$, we thus immediately see that \eqref{eqn_soln_initial} satisfies the initial conditions.
\end{proof}

Suppose $m \in \mathbb{C}[x, x^{-1}]$. Define
\begin{gather*}
    \mathcal{F}_m(n, t) := [x^n]\Big((m, x^t) \Big). 
\end{gather*}
Then $\mathcal{F}_m(n, t)$ is a solution to the wave equation; we call it the \textit{fundamental solution} associated to $m$. Note that if $h_1, h_2$ is a basis of $\mathbb{C}[x, x^{-1}]$ as a module over $\mathbb{C}[x + x^{-1}]$ with dual basis $m_1, m_2$, then $\mathcal{F}_{m_1}(n, t)$ is the solution to the wave equation with $h_1$-initial condition $g_1 = \delta_0(n)$, and $h_2$-initial condition $g_2 = 0$. 

Suppose we take $h_1 = 1, h_2 = \frac{x - x^{-1}}{2}$. Then $m_1 = \frac{x - x^{-1}}{2}, m_2 = 1$. We refer to these as a \textit{standard initial conditions}. Then
\begin{align}
    (m_1, x^t) &= \frac{x^t + x^{-t}}{2} \notag \\
    (m_2, x^t) &= \frac{x^{t} -x^{-t}}{x - x^{-1}} = \begin{cases}
        x^{t-1} + x^{t-3} + \dots + x^{-(t-3)} + x^{-(t-1)} & t \geq 1 \\
        0 & t = 0 \\
        - (x^{-t-1} + x^{-t-3} + \dots + x^{-(-t-3)} + x^{-(-t-1)}) & t \leq -1. 
    \end{cases} \label{eqn_fundamental_soln}
\end{align}
Notice that for this choice of $h_1, h_2$ we have $\mathcal{F}_{m_1}(n, -t) = \mathcal{F}_{m_1}(n, t)$ and $\mathcal{F}_{m_2}(n, -t) = -\mathcal{F}_{m_2}(n, t)$. Very explicitly, specializing Theorem \ref{thm_soln} to this choice of $h_1, h_2$, we find that the unique function $u(n, t)$ such that $u(n, 0) = f(n)$ and $\frac{u(n, 1) - u(n, -1)}{2} = g(n)$ is given by:
\begin{align}
    u(n, t) = &\frac{1}{2}(f(n-t) + f(n+t)) \notag \\
    & + \textnormal{sgn}(t) \big(g(n+|t|-1) + g(n+|t|-3) + \dots + g(n-(|t|-3)) + g (n-(|t|-1)) \big). \label{eqn_soln_std}
\end{align}

\subsection{Finite speed of propagation and d'Alembert solution} \label{sec_d'alembert}
The explicit form of the solution given in \eqref{eqn_soln_initial} has several immediate consequences. The following tells us that waves propagate at ``speed one.''

\begin{corollary}[Finite speed of propagation] \label{cor_finite_speed}
    Suppose $u(n, t)$ is a solution to the wave equation with standard initial conditions $f, g$. Suppose $f$ is supported on $|n| \leq r$ and $g$ is supported on $|n| \leq r+1$. Then $u(n, t)$ is supported on $|n| \leq r + |t|$ for all $t$. 
\end{corollary}
\begin{proof}
    The hypotheses tell us that $f^{[x]}$ is degree at most $r$, and $g^{[x]}$ is degree at most $r+1$. On the other hand, $(m_1, x^t)$ is degree $|t|$, and $(m_2, x^t)$ is degree $|t|-1$. Therefore, $f^{[x]} \cdot (m_1, x^t) + g^{[x]} \cdot (m_2, x^t)$ is degree at most $r + |t|$.  
\end{proof}

We also have the following analogue of d'Alembert's solution to the one-dimensional wave equation.

\begin{corollary}[d'Alembert solution] \label{cor_d'alembert}
    The function $u(n, t)$ is a solution to the wave equation if and only if it can be written as
    \begin{gather*}
        u(n, t) = w_{\infty}(n-t) + w_{-\infty}(n+t)
    \end{gather*}
    for some $w_\infty, w_{-\infty}: \mathbb{Z} \to \mathbb{C}$. If we have $w_{\infty}(n-t) + w_{-\infty}(n+t) = \tilde{w}_\infty(n-t) + \tilde{w}_{-\infty}(n+t)$, then $w_{\infty}(n-t) - \tilde{w}_\infty(n-t) = w_{-\infty}(n+t) - \tilde{w}_{-\infty}(n+t)$ is a function which is constant on $x + t \equiv 0 \mod 2$ and on $x + t \equiv 1 \mod 2$.
\end{corollary}
\begin{proof}
    First off, it is immediate that any function of the form $w_\infty(n-t) + w_{-\infty}(n+t)$ solves \eqref{eqn_discrete_wave_intro}. It is also clear that any function on $\mathbb{Z}^2$ which can be expressed as both a function of $n+t$ and $n-t$ must be constant on any coset of the sublattice generated by $(1, 1)$ and $(1, -1)$. This sublattice is simply the locus $x + t \equiv 0 \mod 2$, and it is index two inside of $\mathbb{Z}^2$.

    Now suppose $u(n, t)$ is the solution to the wave equation with $u(n, 0) = f(n)$ and $\frac{u(n, 1) - u(n, -1)}{2} = g(n)$. First we define
    \begin{gather*}
        w^f_\infty(k) := \frac{f(k)}{2} =: w^f_{-\infty}(k).
    \end{gather*}
    If $g(n) \equiv 0$, then by \eqref{eqn_soln_std}, we have that $u(n, t) = w^f_\infty(n-t) + w^f_{-\infty}(n+t)$.
    
    We now wish to also derive the d'Alembert solution in case $f(n) \equiv 0$. In that case
    \begin{gather*}
        u(n, t) = [x^n]\Big( g^{[x]} \cdot (1, x^t) \Big)
    \end{gather*}
    We define $g_+, g_-: \mathbb{Z} \to \mathbb{C}$ via
    \begin{align*}
        g_+^{[x]} &:= \frac{g(0)}{2} x^0 + \sum_{\ell \geq 1} g(\ell) x^\ell \\
        g_-^{[x]} &:= \frac{g(0)}{2} x^0 + \sum_{\ell \leq -1} g(\ell) x^\ell.
    \end{align*}
    Then $g = g_+ + g_-$. We shall also make use of the identity
    \begin{align*}
        \frac{1}{x - x^{-1}} &= -(x + x^3 + x^5 + \dots) = -M_+^{[x]} \\
        &= x^{-1} + x^{-3} + x^{-5} + \dots = M_-^{[x]},
    \end{align*}
    where $M_+(n) = 1$ if $n \geq 1$ and odd, and zero otherwise, and $M_-(n) = 1$ if $n \leq -1$ and odd, and zero otherwise. We shall essentially use these formal identities as a bookkeeping tool to arrive at a formula which may ultimately be easily verified. 
    
    We therefore have that
    \begin{align*}
        [x^n]\Big( g^{[x]} \cdot (1, x^t) \Big) &= [x^0] \Big( x^{-n} (g_+^{[x]} + g_-^{[x]}) \frac{x^t - x^{-t}}{x - x^{-1}} \Big) \\
        &= [x^0] \Big( x^{-n+t} \frac{g_+^{[x]} + g_-^{[x]}}{x - x^{-1}} - x^{-n-t} \frac{g_+^{[x]} + g_-^{[x]}}{x - x^{-1}} \Big).
    \end{align*}
    We can compute for example that
    \begin{align*}
        [x^0] \Big(x^{-k} \frac{g_+^{[x]}}{x - x^{-1}} \Big) &= [x^0] \Big( x^{-k} g_+^{[x]} \cdot (-M_+^{[x]}) \Big) \\
        &= \Big(g_+ * (-M_+) \Big) (k) \\
        &= \begin{cases}
            -\frac{g(0)}{2} - \sum_{\ell = 2, \ \ell \textnormal{ even}}^{k-1} g(\ell) & k \geq 0, \textnormal{ odd} \\
            -\sum_{\ell = 1, \ \ell \textnormal{ odd}}^{k-1} g(\ell) & k \geq 0, \textnormal{ even} \\
            0 & \textnormal{otherwise}.
        \end{cases}
    \end{align*}
    Similar computations reveal that if we let
    \begin{align*}
        w^g_{\infty}(k) &:= \begin{cases}
            -\textnormal{sgn}(k) \Big( \frac{g(0)}{2} + \sum_{\ell = 2, \ \ell \textnormal{ even}}^{|k|-1} g\big(\textnormal{sgn}(k) \cdot \ell \big) \Big) & k \textnormal{ odd} \\
            -\textnormal{sgn}(k) \sum_{\ell = 1, \ \ell \textnormal{ odd}}^{|k|-1} g(\textnormal{sgn}(k) \cdot \ell) & k \textnormal{ even}
        \end{cases} \\
        w_{-\infty}^g(k) &:=  \begin{cases}
            \textnormal{sgn}(k) \Big( \frac{g(0)}{2} + \sum_{\ell = 2, \ \ell \textnormal{ even}}^{|k|-1} g\big(\textnormal{sgn}(k) \cdot \ell \big) \Big) & k \textnormal{ odd} \\
            \textnormal{sgn}(k) \sum_{\ell = 1, \ \ell \textnormal{ odd}}^{|k|-1} g(\textnormal{sgn}(k) \cdot \ell) & k \textnormal{ even}
        \end{cases},
    \end{align*}
    then 
    \begin{align*}
        w_\infty^g(n-t) + w_{-\infty}^g(n+t) &= \textnormal{sgn}(t) \big(g(n-|t|+1) + g(n-|t|+3) + \dots + g(n+|t|-1) \big) \\
        &= [x^n]\Big( g^{[x]} \cdot (1, x^t) \Big).
    \end{align*}
    If we therefore let $w_\infty = w^f_\infty + w^g_\infty$ and $w_{-\infty} = w^f_{-\infty} + w^g_{-\infty}$, then we have that $u(n, t) = w_\infty(n-t) + w_{-\infty}(n+t)$. 
\end{proof}

Though we did not use it in the proof, we briefly remark the d'Alembert-like identity
\begin{gather*}
    (x - y^{-1})(x - y) = x(x + x^{-1} - y - y^{-1}).
\end{gather*}
Because $x$ is a unit in the group algebra, we have that \eqref{eqn_discrete_wave_intro} is in turn equivalent to $f^{[x, y]}$ being annihilated by multiplication with $(x - y^{-1})(x -y)$. 

We also have a ``spectral'' version of the d'Alembert solution.

\begin{proposition} \label{prop_spectral}
    A function $g(\xi, t)$ satisfies \eqref{spectral_wave_eqn} if and only if it can be written as
    \begin{gather} \label{spectral_soln}
        g(\xi, t) = g_\infty(\xi) \xi^{t} + g_{-\infty}(\xi) \xi^{-t} + \delta_{1}(\xi) [a_1 t + b_1] + (-1)^t \delta_{-1}(\xi) [a_{-1} t + b_{-1}].
    \end{gather}
    If $g_{\infty}(\xi) \xi^{t} + g_{-\infty}(\xi) \xi^{-t} = \tilde{g}_{\infty}(\xi) \xi^{t} + \tilde{g}_{-\infty}(\xi) \xi^{-t}$, then $[g_\infty(\xi) - \tilde{g}_\infty(\xi)] \xi^t = [\tilde{g}_{-\infty}(\xi) - g_{-\infty}(\xi)] \xi^{-t} = c_1 \delta_{1}(\xi) + c_{-1} \delta_{-1}(\xi)$. 
\end{proposition}

\begin{proof}
    It is immediate that any function of the form in \eqref{spectral_soln} solves \eqref{spectral_wave_eqn}. We should remark that functions of the form $A(t) = a_1 t + b_1$ are the functions satisfying $A(t-1) + A(t+1) = 2 A(t)$, and functions of the form $B(t) = (-1)^t (a_{-1} t + b_{-1})$ are the functions satisfying $A(t-1) + A(t+1) = -2 A(t)$.
    
    We also note that if $a(\xi) \xi^t = b(\xi) \xi^{-t}$ for all $\xi, t$, then $a(\xi) \xi^{2 t} = b(\xi)$. Unless $\xi^{2 t} = 1$ for all $t \in \mathbb{Z}$, we must have that $a(\xi) = b(\xi) = 1$. On the other hand $\xi^{2 t} = 1$ for all $t$ is equivalent to $\xi = \pm 1$. 

    Now let's prove that if $g$ solves \eqref{spectral_wave_eqn}, then it can be expressed in the claimed form. First we remark that if $g$ solves the spectral wave equation, then it is uniquely determined by its value at $t = 0$ and $t = 1$. This follows inductively by \eqref{spectral_wave_eqn}. On the other hand, we claim that given any function $h: \mathbb{C}^\times \times \mathbb{Z} \to \mathbb{C}$ defined only at $t = 0$ and $t = 1$ extends to a function of the form in \eqref{spectral_soln}. Since functions of this form solve the wave equation, we get that solutions to the wave equation are all of this form.

    Suppose we are given two functions $a(\xi)$ and $b(\xi)$. For now suppose $\xi \neq \pm 1$. We claim that there exists a unique function $g(\xi, t) = g_\infty(\xi) \xi^{t} + g_{-\infty} (\xi) \xi^{-t}$ outside of $\xi = \pm 1$ such that $g(\xi, 0) = a(\xi)$ and $g(\xi, 1) = b(\xi)$. We must have $a(\xi) = g_\infty(\xi) + g_{-\infty}(\xi)$ and $b(\xi) = g_\infty(\xi) \xi + g_{-\infty}(\xi) \xi^{-1}$. In matrix form
    \begin{gather*}
        \begin{bmatrix}
            1 & 1 \\
            \xi & \xi^{-1}
        \end{bmatrix} \begin{bmatrix}
            g_\infty(\xi) \\
            g_{-\infty}(\xi)
        \end{bmatrix} = \begin{bmatrix}
            a(\xi) \\
            b(\xi)
        \end{bmatrix}
    \end{gather*}
    Thus, so long as $\xi - \xi^{-1} \neq 0$, we get that $g_\infty(\xi)$ and $g_{-\infty}(\xi)$ are uniquely determined by $a(\xi)$ and $b(\xi)$. Note that $\xi - \xi^{-1} = 0$ if and only if $\xi = \pm 1$. 
    
    For $\xi = \pm 1$, we must solve a corresponding linear recurrence relation and, as discussed in the first paragraph of the proof, the corresponding general solutions are of stated form.
\end{proof}

\begin{remark} 
It may seem that there is a discrepancy between Corollary \ref{cor_d'alembert} and Proposition \ref{prop_spectral}. This may be explained as follows. Any time we have a solution to the wave equation \eqref{eqn_discrete_wave_intro} for which we may reasonably take a discrete Fourier-Laplace transform in an analytic sense, we necessarily get a solution to \eqref{spectral_wave_eqn} of the form $\hat{w}_\infty(\xi) \xi^{t} + \hat{w}_{-\infty}(\xi) \xi^{-t}$. On the other hand, given an integrable function of the form in \eqref{spectral_soln} for which we may take Fourier series (by restricting to $S^1$), we may freely change the value of the function at $\pm 1$ without changing the Fourier series, so these ``extra terms'' are somewhat irrelevant. If we instead consider $\delta_1(\xi)$ as a distribution, then its Fourier series corresponds to the constant function on $\mathbb{Z}$. Taking the Fourier series of $\delta_1(\xi)[a_1 t + b_1]$ gives the function $h: \mathbb{Z} \times \mathbb{Z} \to \mathbb{C}$ given by $h(n, t) = a_1 t + b_1$. This function can indeed be written as a sum of a left and right moving wave: $h(n, t) = \frac{a_1}{2} ((t + n) + (t - n)) + b_1$. Let $h_1(n, t) = t+n$ and $h_2(n, t) = t-n$. Restricting to $\xi \in S^1$, we get that $\hat{h}_1(\xi, t) = \sum_{n = -\infty}^{\infty} (t+n) \xi^n = \xi^{-t} \sum_{n = -\infty}^{\infty} (t + n) \xi^{t + n}$ which we can think of as the distribution $i \xi^{-t} \delta'_{1}(\xi)$. Similarly we can think of $\hat{h}_2(\xi, t)$ as $-i \xi^{t} \delta'_{-1}(\xi)$. As a distribution we get that $i \frac{a_1}{2}(\xi^{-t} - \xi^{t}) \delta'_{1}(\xi) + b_1 \delta_{1}(\xi) = \delta_1(\xi) [a_1 t + b_1]$. Something similar happens for $\delta_{-1}$. 
\end{remark}

\subsection{Chebyshev polynomials and orthogonality} \label{sec_chebyshev}

We know that every symmetric polynomial in $\mathbb{C}[x, x^{-1}]$ can in turn be expressed as a polynomial in $x + x^{-1}$ or equivalently as a polynomial in $\frac{x + x^{-1}}{2}$. On the other hand, each fundamental solution is a sequence of symmetric polynomials of growing degree. The Chebyshev polynomials of the first and second kind $T_t$ and $U_t$ are exactly the polynomials expressing the fundamental solutions in \eqref{eqn_fundamental_soln} as polynomials in $\frac{x + x^{-1}}{2}$, i.e.
\begin{align*}
    T_t( \frac{x + x^{-1}}{2}) &= \frac{ x^t + x^{-t}}{2} \\
    U_{t-1} (\frac{ x + x^{-1}}{2}) &= \frac{x^{t} - x^{-t}}{x - x^{-1}}.
\end{align*}
The Chebyshev polynomials satisfy the recurrence relation $A_{t+1}(z) + A_{t-1}(z) = 2 z A_t(z)$, which, by setting $z = \frac{\xi + \xi^{-1}}{2}$, is simply a reformulation of \eqref{spectral_wave_eqn} for solutions $g(\xi, t)$ such that $g(\cdot, t)$ is always symmetric.

More generally, given $h(x) \in \mathbb{C}[x, x^{-1}]$ for each $t \in \mathbb{Z}$, we may define a corresponding Chebyshev polynomial $\textnormal{Ch}_t^h(z)$
\begin{gather}
    \textnormal{Ch}_t^h (\frac{x + x^{-1}}{2}) = \frac{h(x) x^t - h(x^{-1}) x^{-t}}{x - x^{-1}}. \label{eqn_chebyshev}
\end{gather}
Thus $T_t$ corresponds to $h = \frac{x - x^{-1}}{2}$, $U_{t-1}$ corresponds to $h = 1$. Other well-studied examples are the Chebyshev polynomials of the third and fourth kind, often denoted $V_t$ and $W_t$, which correspond to $h = x - 1$ and $h = x + 1$, respectively.

Using the following theorem, we shall see that the orthogonality of the Chebyshev polynomials (and related polynomials) follows directly from the fact that they are expressible in a form as in \eqref{eqn_chebyshev}. After discovering the below result, we realized that it is essentially a reformulation of the theory of Bernstein-Szegö polynomials. See \cite{szego, li_sole}. We include the proof as it illustrates several ideas which will be useful in later sections.

\begin{theorem} \label{thm_orthogonality}
    Let $h(\xi) = \sum_{i = n}^m a_i \xi^i \in \mathbb{R}[\xi, \xi^{-1}]$. Suppose $a_n, a_m \neq 0$. Define
    \begin{gather*}
        f_k(\xi) = \frac{h(\xi) \xi^{k} - h(\xi^{-1}) \xi^{-k}}{\xi - \xi^{-1}}.
    \end{gather*}
    Let $d \theta$ denote the Lebesgue measure on $S^1 \subset \mathbb{C}^\times$ corresponding to arclength (i.e. total mass $2 \pi$). Let $d \mu$ be the measure on $\mathbb{C}^\times$ defined by
    \begin{gather*}
        d\mu = \frac{(\xi - \xi^{-1})(\xi^{-1} - \xi)}{h(\xi) h(\xi^{-1})} d \theta.
    \end{gather*}
    \begin{enumerate}
        \item Suppose $\frac{h(\xi)}{h(\xi^{-1})}$ has no poles on the closed unit disk $\overline{\mathbb{D}}$. Then $f_k(\xi)$ with $k \geq 0$ are orthogonal on $\mathbb{C}^\times$ with respect to $d \mu$. If $\frac{h(\xi)}{h(\xi^{-1})}$ has no poles on $\hat{\mathbb{C}} \setminus \mathbb{D}$, then $f_k(\xi)$ with $k \leq 0$ are orthogonal with respect to $d \mu$.
        \item Suppose
        \begin{gather*}
            h(\xi) = (1 - a^{-1} \xi^{-1}) (1 + b^{-1} \xi^{-1}) \xi
        \end{gather*}
        with $1 < a < \infty$ and $0 < b < 1$. Then $f_k(\xi)$ with $k \geq 0$ are orthogonal on $\mathbb{C}^\times$ with respect to the measure
        \begin{gather*}
            d \mu + \frac{4 \pi (1-b^2)}{(1 + a^{-1} b^{-1})(1 + a^{-1} b)} \delta_{-b}.
        \end{gather*}
    \end{enumerate}
\end{theorem}

\begin{proof}
    We begin with Case (1). Note that $f_k(\xi)$ is always real-valued on $S^1$. We have
    \begin{align*}
        \int_{\mathbb{C}^\times} f_k(\xi) f_n(\xi) d \mu &= \int_{S^1} \frac{h(\xi) \xi^k - h(\xi^{-1}) \xi^{-k}}{\xi - \xi^{-1}} \frac{h(\xi) \xi^{n} - h(\xi^{-1}) \xi^{-n}}{\xi - \xi^{-1}} \frac{(\xi - \xi^{-1})(\xi^{-1} - \xi)}{h(\xi) h(\xi^{-1})} d \theta \\
        &= \int_{S^1} \xi^{k-n} + \xi^{n-k} - \frac{h(\xi^{-1})}{h(\xi)} \xi^{-n-k} - \frac{h(\xi)}{h(\xi^{-1})} \xi^{n+k} d \theta
    \end{align*}
    If $k \neq n$, then the first two terms in the integral vanish. Furthermore note that $\int_{S^1} \frac{h(\xi^{-1})}{h(\xi)} \xi^{-n-k} d \theta = \int_{S^1} \frac{h(\xi)}{h(\xi^{-1})} \xi^{n+k} d \theta$. Additionally we have the relation
    \begin{gather*}
        \int_{S^1} \frac{h(\xi)}{h(\xi^{-1})} \xi^{n+k} d \theta = \frac{1}{i} \oint_{S^1} \frac{h(\xi)}{h(\xi^{-1})} \xi^{n+k-1} d \xi.
    \end{gather*}
    If $n, k \geq 0$ and $n \neq k$, then $n + k - 1 \geq 0$. We are assuming that $\frac{h(\xi)}{h(\xi^{-1})}$ has no poles on $\overline{\mathbb{D}}$, so by the residue theorem we conclude that the above integral is zero. One can carry out an analogous argument in case $\frac{h(\xi)}{h(\xi^{-1})}$ has no poles on $\hat{\mathbb{C}} \setminus \mathbb{D}$.

    Now we consider Case (2). We have
    \begin{gather*}
        \frac{h(\xi)}{h(\xi^{-1})} = \frac{(1 - a^{-1} \xi^{-1})(1 + b^{-1} \xi^{-1}) \xi}{(1 - a^{-1} \xi) (1 + b^{-1} \xi) \xi^{-1}} = \frac{(1 - a^{-1} \xi^{-1})(1 + b^{-1} \xi^{-1}) \xi^2 b}{1 - a^{-1} \xi} \frac{1}{\xi + b}.
    \end{gather*}
    Therefore by the residue theorem (with $n \neq k$)
    \begin{gather*}
        \int_{\mathbb{C}^\times} f_k(\xi) f_n(\xi) d \mu = -4 \pi \frac{(1 + a^{-1} b^{-1})(1 - b^{-2}) (-b)^{n+k-1} b}{1 + a^{-1} b}.
    \end{gather*}
    On the other hand we have
    \begin{gather*}
        \int_{\mathbb{C}^\times} f_k(\xi) f_n(\xi) d \delta_{-b}(\xi) = \frac{(1 + a^{-1} b^{-1})^2 (1 - b^{-2})^2 (-b)^{n+k}}{(-b + b^{-1})^2},
    \end{gather*}
    from which the claim follows. 
\end{proof}

We have a branched double cover $\mathbb{C}^\times \to \mathbb{C}$ via $\xi \mapsto \frac{ \xi + \xi^{-1}}{2}$ with branch points $\xi = \pm 1$. Any function on $\mathbb{C}^\times$ which is invariant under $\xi \mapsto \xi^{-1}$ is mapped to a well-defined function on $\mathbb{C}$ via this map. The circle $S^1$ is mapped to $[-1, 1]$. In terms of $\theta$, we have that $\theta = \arccos(z)$, so that $d \theta = \frac{1}{\sqrt{1 - z^2}} d z$. 

The interval $(0, 1)$ is mapped to $(1, \infty)$, and the interval $(-1, 0)$ is mapped to $(-\infty, -1)$. By pushing forward the measure $d \mu$ in the Proposition to $[-1, 1]$ (and in Case (2) also pushing forward the measure $\delta_{-b}$), we get natural families of orthogonal polynomials on $[-1, 1]$ (or on $\mathbb{R}$) with respect to different measures. We now discuss several special cases.

\begin{example} \label{cheb_first_kind}
    Let $h(\xi) = \frac{\xi - \xi^{-1}}{2}$. Then 
    \begin{gather*}
        d \mu = 4 d \theta = \frac{4}{\sqrt{1 - z^2}} dz,
    \end{gather*}
    which is the orthogonality measure for the Chebyshev polynomials of the first kind.
\end{example}

\begin{example}\label{cheb_second_kind}
    Let $h(\xi) = 1$. Then 
    \begin{gather*}
        d \mu = (\xi - \xi^{-1})(\xi^{-1} - \xi) d \theta = \frac{4 (1 - z^2)}{\sqrt{1 - z^2}} d z = 4 \sqrt{1 - z^2} dz
    \end{gather*}
    which is the orthogonality measure for the Chebyshev polynomials of the second kind. This measure is exactly the semicircle law.
\end{example}

\begin{example} \label{example_kesten_mckay}
    Let $h(\xi) = \xi - q^{-1} \xi^{-1}$ with $q \geq 1$. Then
    \begin{gather*}
        d \mu = \frac{(\xi - \xi^{-1})(\xi^{-1} - \xi)}{(\xi - q^{-1} \xi^{-1})(\xi^{-1} - q^{-1} \xi)} d \theta = \frac{4 (1 - z^2)}{(1 + q^{-2}) - q^{-1}(4 z^2 - 2)} \frac{1}{\sqrt{1 - z^2}} dz = \frac{4 q^2 \sqrt{1 - z^2}}{(q+1)^2 - 4q z^2} dz.
    \end{gather*}
    This is exactly the Kesten--McKay law for $(q+1)$-regular graphs, or equivalently the spherical Plancherel measure (with respect to the Satake parameters in $\xi$-coordinates, and with respect to Hecke eigenvalue in $z$-coordinates) for $\textnormal{PGL}(2, F)$ where $F$ is a non-archimedean local field with residue field of order $q$. If $q = 1$, this reduces to Example \ref{cheb_first_kind}, and as $q \to \infty$ this converges to Example \ref{cheb_second_kind}. We shall use the notation $F_t$ for the corresponding Chebyshev polynomials. They have the property that $F_t(\frac{A}{2 \sqrt{q}}) \delta_o$ is supported on a sphere of radius $t$ centered at $o$, where $A$ is the adjacency operator, and $o$ is some vertex in the tree. We shall say more about this example in Section \ref{sec_wave_eqn_tree}.
\end{example}

\begin{example}
    Suppose $h(\xi) = 1 - \xi$. Then
    \begin{gather*}
        d \mu = \frac{(\xi - \xi^{-1})(\xi^{-1} - \xi)}{(1 - \xi)(1 - \xi^{-1})} d \theta = \frac{4 (1 - z^2)}{2(1 - z)} \frac{1}{\sqrt{1 - z^2}} dz = 2 \frac{\sqrt{1 + z}}{\sqrt{1 - z}} dz.
    \end{gather*}
    This is the orthogonality measure for the Chebyshev polynomials of the third kind, often denoted $V_t$.
\end{example}

\begin{example}
    Suppose $h(\xi) = 1 + \xi$. Then 
    \begin{gather*}
        d \mu = \frac{(\xi - \xi^{-1})(\xi^{-1} - \xi)}{(1 + \xi)(1 + \xi^{-1})} d \theta = \frac{4(1 - z^2)}{2 (1 + z)} \frac{1}{\sqrt{1 - z^2}} dz = 2 \frac{\sqrt{1 - z}}{\sqrt{1 + z}} dz.
    \end{gather*}
    This is the orthogonality measure for the Chebyshev polynomials of the fourth kind, often denoted $W_t$.
\end{example}

\begin{example} \label{marchenko_pastur}
    Suppose $h(\xi) = (1 + b^{-1} \xi^{-1}) \xi$. Then
    \begin{gather*}
        d \mu = \frac{(\xi - \xi^{-1})(\xi^{-1} - \xi)}{(1 + b^{-1} \xi^{-1})(1 + b^{-1} \xi)} = \frac{4(1 - z^2)}{(1 + b^{-2}) +  2 b^{-1} z} \frac{1}{\sqrt{1 - z^2}} dz = \frac{4 \sqrt{1 - z^2}}{(1 + b^{-2}) + 2 b^{-1} z} dz. 
    \end{gather*}
    If we set $z = \frac{x - (1 + b^{-2})}{2 b^{-1}}$, this transforms into
    \begin{gather*}
        \frac{4\sqrt{(x - (1 - b^{-1})^2)((1 + b^{-1})^2-x)}}{b^{-2} x} dx.
    \end{gather*}
    If $b^{-1} \leq 1$, then we are in Case (1). If $b^{-1} > 1$, then we are in Case (2) (with $a^{-1} = 0$) and we need to add on the Dirac delta:
    \begin{gather*}
        4 \pi (1 - b^2) \delta_{-b}(\xi) = 4 \pi (1 - b^2) \delta_{-\frac{b + b^{-1}}{2}}(z) = 4 \pi (1 - b^2) \delta_{0}(x). 
    \end{gather*}
    This corresponds exactly to the Marchenko--Pastur law (with parameter $\lambda = b^{-2}$).
\end{example}

\begin{example} \label{ex_biregular}
    Suppose $h(\xi) = (1 - \frac{\xi^{-1}}{\sqrt{pq}}) (1 + \frac{ \sqrt{p} \xi^{-1}}{\sqrt{q}}) \xi$ with $p, q \geq 1$. If $p \leq q$, then we are in Case (1) from Theorem \ref{thm_orthogonality}, and we get the measure
    \begin{gather*}
        d \mu = \frac{4 \sqrt{1 - z^2}}{((1 + \frac{p}{q}) + 2 \frac{\sqrt{p}}{\sqrt{q}} z) ((1 + \frac{1}{pq}) - 2 \frac{1}{\sqrt{pq}} z)} dz.
    \end{gather*}
    In case $p > q$, we also get the term
    \begin{gather}
        \frac{4 \pi (1 - \frac{q}{p})}{(1 + \frac{1}{q})(1 + \frac{1}{p})} \delta_{\frac{-(p+q)}{2 \sqrt{pq}}}(z). \label{eqn_biregular_dirac_delta}
    \end{gather}
    This corresponds to the analogue of the Kesten--McKay law for the $(p+1, q+1)$-biregular tree (with respect to a certain explicit operator $B_q$ expressible as a degree two even polynomial in the adjacency operator acting on functions supported on the vertices of degree $q+1$). In the special case that $(p, q) = (r, r^3)$ with $r$ a power of a prime, this in turn corresponds to the spherical Plancherel measure for the group $SU(3, E/F)$ where $F$ is a non-archimedean local field whose residue field is $\mathbb{F}_r$, and $E$ is an unramified quadratic extension. If $p, q \to \infty$, but $p/q$ remains fixed, then this converges to Example \ref{marchenko_pastur} (in fact all of the above examples are special cases or limiting cases of this example). We shall use the notation $H_t$ for the corresponding Chebyshev polynomials. They have the property that $H_t(B_q) \delta_o$ is supported on a sphere of radius $t$ around $o$, assuming $o$ is a vertex of degree $q+1$.  We shall discuss this further in Section \ref{sec_big_biregular_tree}. 
\end{example}

\subsection{Conservation of energy} \label{sec_energy}

We now show that our discrete wave equation conserves an appropriate energy form. Suppose we have initial conditions determined by $h_1, h_2$ with dual basis $m_1, m_2$. Then \eqref{eqn_soln_initial} tells us that if $u(n, t)$ is a solution to the wave equation,
\begin{gather}
    \begin{bmatrix}
        [y^t](u^{[x, y]} \cdot h_1(y^{-1}))) \\
        [y^t](u^{[x, y]} \cdot h_2(y^{-1}))
    \end{bmatrix} = \begin{bmatrix}
        (m_1 h_1, x) & (m_2 h_1, x) \\
        (m_1 h_2, x) & (m_2 h_2, x)
    \end{bmatrix} \begin{bmatrix}
        [y^{t-1}]\Big(u^{[x, y]} \cdot h_1(y^{-1}) \Big) \\
        [y^{t-1}] \Big(u^{[x, y]} \cdot h_2(y^{-1}) \Big)
    \end{bmatrix}. \label{eqn_general_time_one_step}
\end{gather}
The matrix appearing in \eqref{eqn_general_time_one_step} is thus the time one wave propagator with respect to the initial conditions determined by $h_1, h_2$. 

For the rest of this section we shall work with the standard initial conditions $h_1 = 1, h_2 = \frac{x - x^{-1}}{2}$ as this will result in the nicest formulas. Let $u(n, t)$ be such that $u(n, 0) = f(n)$ and $\frac{u(n, 1) - u(n, -1)}{2} = g(n)$. The matrix in \eqref{eqn_general_time_one_step} then becomes
\begin{gather}
    \mathcal{U} = \begin{bmatrix}
        \frac{x + x^{-1}}{2} & 1 \\
        \frac{x^2 - 2 + x^{-2}}{4} & \frac{x + x^{-1}}{2}
    \end{bmatrix} = \begin{bmatrix}
        z & 1 \\
        z^2 - 1 & z
    \end{bmatrix}, \hspace{5mm} \mathcal{U}^t = \begin{bmatrix}
        T_t(z) & U_{t-1}(z) \\
        (z^2 - 1) U_{t-1}(z) & T_t(z)
    \end{bmatrix} \label{eqn_one_time_step_prop_U}
\end{gather}
where $z = \frac{x + x^{-1}}{2}$, $T_t(z)$ are the Chebyshev polynomials of the first kind, and $U_t(z)$ are the Chebysehv polynomials of the second kind. We thus in particular see that $u(n, t) = [x^n] \Big(T_t(\frac{x + x^{-1}}{2}) \cdot f^{[x]} + U_{t-1}(\frac{x + x^{-1}}{2}) \cdot g^{[x]} \Big)$, which is simply a reformulation of \eqref{eqn_soln_std}.

We have that
\begin{gather}
    \mathcal{U} = \begin{bmatrix}
        \frac{2}{x - x^{-1}} & -\frac{2}{x - x^{-1}} \\
        1 & 1
    \end{bmatrix} \begin{bmatrix}
        x & 0 \\
        0 & x^{-1}
    \end{bmatrix}
    \begin{bmatrix}
        \frac{x - x^{-1}}{2} & 1 \\
        \frac{x^{-1} - x}{2} & 1
    \end{bmatrix}. \label{eqn_eigen_decomp_U}
\end{gather}
We now switch to the analytic perspective. We replace $x$ by $\xi$, which we think of as a element of $S^1$. We then think of the above as matrices of functions on $S^1$. Given such a matrix $M$, we define $M^*$ as the conjugate transpose matrix, i.e. if the $(i, j)$ entry of $M$ is $\sum a_k \xi^k$, then the $(j, i)$ entry of $M^*$ is $\sum \overline{a_k} \xi^{-k}$. Let $B$ be the last matrix appearing in \eqref{eqn_eigen_decomp_U}. Then define
\begin{gather*}
    Q := \frac{1}{2} B^* B = \begin{bmatrix}
        -\Big( \frac{\xi - \xi^{-1}}{2} \Big)^2 & 0 \\
        0 & 1
    \end{bmatrix} = \begin{bmatrix}
        1 - z^2 & 0 \\
        0 & 1
    \end{bmatrix}.
\end{gather*}
By construction we have that $\mathcal{U}^* Q \mathcal{U} = Q$. Let $v(n, t) := \frac{u(n, t+1) - u(n, t-1)}{2}$. We then have
\begin{align*}
    & \begin{bmatrix}
        \overline{u^{[\xi]}(\cdot, t)} & \overline{v^{[\xi]}(\cdot, t)}
    \end{bmatrix} Q \begin{bmatrix}
        u^{[\xi]}(\cdot, t) \\
        v^{[\xi]}(\cdot, t)
    \end{bmatrix} = \begin{bmatrix} \overline{f^{[\xi]}} & \overline{g^{[\xi]}}
    \end{bmatrix} (\mathcal{U}^t)^* Q \mathcal{U}^t \begin{bmatrix}
        f^{[\xi]} \\
        g^{[\xi]}
    \end{bmatrix} = \begin{bmatrix} \overline{f^{[\xi]}} & \overline{g^{[\xi]}}
    \end{bmatrix} Q \begin{bmatrix}
        f^{[\xi]} \\
        g^{[\xi]}
    \end{bmatrix}. 
\end{align*}
We thus have that
\begin{gather*}
    \int_{S^1} -|u^{[\xi]}(\cdot, t)|^2 \Big(\frac{\xi - \xi^{-1}}{2} \Big)^2 + |v^{[\xi]}(\cdot, t)|^2 d \theta(\xi) = \int_{S^1} -|f^{[\xi]}|^2 \Big(\frac{\xi - \xi^{-1}}{2} \Big)^2 + |g^{[\xi]}|^2 d \theta(\xi).
\end{gather*}
If we then apply the Plancherel formula, we get:
\begin{proposition} \label{prop_energy_flat}
    Let $u(n, t)$ be a solution to the wave equation. Then
    \begin{align}
        E(t) :&= -\sum_{n = -\infty}^{\infty} \overline{u(n, t)} \cdot \Big( \frac{u(n+2, t) - 2 u(n, t) + u(n-2, t)}{4} \Big) + \sum_{n = -\infty}^\infty \Big|\frac{u(n, t+1)-u(n, t-1)}{2}\Big|^2 \notag \\
        &= \Big\|\frac{u(n+1, t) - u(n-1, t)}{2}\Big\|^2_{\ell^2(\mathbb{Z}, n)} + \Big\|\frac{u(n, t+1) - u(n, t-1)}{2}\Big\|^2_{\ell^2(\mathbb{Z}, n)}  \label{eqn_flat_energy_form}
    \end{align}
    is independent of $t$. 
\end{proposition}

In Corollary \ref{cor_d'alembert}, we showed that every solution to the wave equation can be represented as $u(n, t) = w_\infty(n-t) + w_{-\infty}(n+t)$. The functions $w_\infty, w_{-\infty}$ were not quite unique; they were only unique up to a constant on $2 \mathbb{Z}$ and on $1 + 2 \mathbb{Z}$. However, the functions $w'_\infty(n) := \frac{w_\infty(n+1) - w_\infty(n-1)}{2}$ and $w'_{-\infty}(n) := \frac{w_{-\infty}(n+1) - w_{-\infty}(n-1)}{2}$ are canonically defined regardless of the choice of constant. 
\begin{proposition} \label{prop_energy_left_right}
    Let $u(n, t) = w_\infty(n-t) + w_{-\infty}(n+t)$ be a solution to the wave equation. Then the value of the energy form defined by \eqref{eqn_flat_energy_form} is equal to
    \begin{gather*}
        \|w'_{\infty}(n)\|^2_{\ell^2(\mathbb{Z}, n)} + \|w'_{-\infty}(n)\|^2_{\ell^2(\mathbb{Z}, n)}.
    \end{gather*}
\end{proposition}
\begin{proof}
    This may be checked by a straightforward but slightly tedious computation.
\end{proof}

Several authors have independently derived the invariant energy form \cite{romanov_rudin_95, anker_martinot_pedon_setti, brooks_lindenstrauss, tao_blog}, particularly as it relates to the wave equation on the regular tree to be discussed in a later section.

\subsection{The wave equation on $\mathbb{R}$ revisited} \label{sec_wave_eqn_R}
We briefly explain how we may interpret some of the standard properties of the wave equation on $\mathbb{R}$ as arising through an analogous algebraic treatment to the preceding sections. Here $\mathbb{C}[x, x^{-1}]$ is replaced by $\mathbb{C}[x]$, and $\mathbb{C}[x, x^{-1}]^W$ is replaced by $\mathbb{C}[x]^W = \mathbb{C}[x^2]$ (with $w_0: x \mapsto -x$). The following is easy to show:
\begin{proposition}
    We have that
    \begin{enumerate}
        \item $\mathbb{C}[x]$ is a rank 2 free module over $\mathbb{C}[x^2]$.
        \item We can take as a free basis $1, \frac{x}{2}$.
        \item The vector space $\mathbb{C}[x]/\mathbb{C}[x^2]$ with the natural $W$-action is isomorphic to the regular representation.
    \end{enumerate}
\end{proposition}
    We have a non-degenerate $\mathbb{C}[x^2]$-bilinear pairing on $\mathbb{C}[x]$ taking values in $\mathbb{C}[x^2]$ via
    \begin{gather*}
        (f, g) := \frac{f(x) g(x) - f(-x) g(-x)}{x}.
    \end{gather*}
    This allows us to speak of dual bases. The dual basis of $1, \frac{x}{2}$ is $\frac{x}{2}, 1$.

    We have a natural map $D$ from $\mathbb{C}[x]$ to constant coefficient differential operators on $\mathbb{R}$ induced from sending $x$ to $\frac{\partial}{\partial x}$. Let $\mathcal{F}$ denote the Fourier transform:
    \begin{gather*}
        [\mathcal{F} f](\lambda) = \hat{f}(\lambda) = \int_{-\infty}^\infty f(x) e^{i \lambda x} dx.
    \end{gather*}
    The following is straightforward to show. For simplicity we assume compactly supported initial conditions to guarantee that we can readily take their Fourier transform, but this assumption can be easily dropped if we work with the distributional Fourier transform.

    \begin{proposition}
        Let $h_1, h_2$ be a free basis, with dual basis $m_1, m_2$. Let $g_1, g_2 \in C_c(\mathbb{R})$. The unique solution to the wave equation satisfying
        \begin{align*}
            D\big(h_1(t)\big) u(x, 0) &= g_1(x) \\
            D\big(h_2(t)\big) u(x, 0) &= g_2(x)
        \end{align*}
        is given by
        \begin{gather*}
            u(x, t) = \mathcal{F}^{-1} \Big( \hat{g}_1(\lambda) \cdot (m_1(i \lambda), e^{i \lambda t}) + \hat{g}_2(\lambda) \cdot (m_2(i \lambda), e^{i \lambda t}) \Big)(x). 
        \end{gather*}
    \end{proposition}

For example with respect to the basis $h_1 = 1, h_2 = \frac{x}{2}$, we have
\begin{align*}
    (m_1(i \lambda), e^{i \lambda t}) &= (\frac{i \lambda}{2}, e^{i \lambda t}) = \frac{e^{i \lambda t} + e^{-i \lambda t}}{2} = \cos(\lambda t) = \mathcal{F}( \frac{\delta_t(x) + \delta_{-t}(x)}{2}) \\
    (m_2(i \lambda), e^{i \lambda t}) &= (1, e^{i \lambda t}) = \frac{e^{i \lambda t} - e^{-i \lambda t}}{ i \lambda} = \frac{2\sin(\lambda t)}{\lambda} = \mathcal{F}(\textnormal{sgn}(t) 1_{[-|t|, |t|]}(x)). 
\end{align*}
This implies that the solution to the wave equation $u(x, t)$ satisfying $u(x, 0) = f(x)$ and $\frac{u_t (x, 0)}{2} = g(x)$ is given by
\begin{gather}
    u(x, t) = \frac{f(x + t) + f(x - t)}{2} + \int_{x-t}^{x + t} g(s) ds. \label{eqn_std_soln_R}
\end{gather}
Already from this expression, finite speed of propagation (or, more precisely, propagation at speed exactly 1) is clear. It's also clear from this expression that the compact support assumption on $f$ and $g$ was not necessary.

We now explain how to derive the d'Alembert solution in an analogous way to Corollary \ref{cor_d'alembert}. We wish to write $u(x, t) = w_\infty(x-t) + w_{-\infty}(x+t)$. First define
\begin{gather*}
    w^f_\infty(k)  := \frac{f(k)}{2} =: w^f_{-\infty}(k).
\end{gather*}
Then if $g(x) = 0$, clearly $u(x, t) = w^f_{\infty}(x-t) + w^f_{-\infty}(x+t)$ by \eqref{eqn_std_soln_R}. We now wish to derive the d'Alembert solution in case $f(x) = 0$. In that case
\begin{gather*}
    u(x, t) = \int_{x-t}^{x+t} g(s) ds = \mathcal{F}^{-1} \Big( \hat{g}(\lambda) \cdot  (1, e^{i \lambda t}) \Big)(x) = \mathcal{F}^{-1} \Big( e^{-i \lambda x} \hat{g}(\lambda) \cdot (1, e^{i \lambda t}) \Big)(0).
\end{gather*}
Write
\begin{gather*}
    g(x) = g(x) \cdot 1_{[0, \infty)} + g(x) \cdot 1_{(-\infty, 0)} = g_+(x) + g_-(x).
\end{gather*}
We have that
\begin{gather*}
    e^{-i \lambda x} \hat{g}(\lambda) \cdot (1, e^{i \lambda t}) = e^{-i\lambda (x - t)} \frac{\hat{g}_+(\lambda) + \hat{g}_-(\lambda)}{i \lambda} - e^{-i\lambda(x+t)} \frac{\hat{g}_+(\lambda) + \hat{g}_-(\lambda)}{i \lambda}.
\end{gather*}
We also have
\begin{align*}
    \frac{1}{i \lambda} &= - \int_{0}^\infty e^{i \lambda x} dx = \mathcal{F} \Big( - 1_{[0, \infty)} \Big) \\
    &= \int_{-\infty}^0 e^{i \lambda x} dx = \mathcal{F} \Big( 1_{(-\infty, 0)} \Big).
\end{align*}
Therefore, 
\begin{align*}
    \mathcal{F}^{-1} \Big( e^{-i \lambda k} \frac{\hat{g}_+(\lambda)}{i \lambda} \Big) &= \mathcal{F}^{-1} \Big(e^{-i \lambda k} \mathcal{F} (g_+(x))(\lambda) \cdot \mathcal{F} (-1_{[0, \infty)})(\lambda) \Big) \\
    &= \Big(g_+ * (-1_{[0, \infty)}) \Big) (k) \\
    &= \begin{cases} 
    -\int_0^k g(s) ds & k \geq 0 \\
    0 & \textnormal{otherwise}.
    \end{cases}
\end{align*}
Similar computations reveal that if
\begin{align*}
    w^g_\infty(k) &:= - \int_0^k g(s) ds \\
    w^g_{-\infty}(k) &:= \int_{0}^k g(s) ds
\end{align*}
then
\begin{gather*}
    w^g_{\infty} (x-t) + w^g_{-\infty} (x + t) = \int_{x-t}^{x + t} g(s) ds. 
\end{gather*}
We can therefore set $w_\infty = w^f_\infty + w^g_\infty$ and $w_{-\infty} = w^f_{-\infty} + w^f_{-\infty}$. Note that if we have $w_\infty(x-t) + w_\infty(x+t) = \tilde{w}_\infty(x-t) + \tilde{w}_\infty(x+t)$, then $w_\infty(x-t) - \tilde{w}_\infty(x-t) = \tilde{w}_{-\infty}(x+t) - w_{-\infty}(x+t)$ must be a constant as this is the only function expressible as a function of both $x-t$ and $x+t$.

Finally we discuss conservation of energy. Let $h_1 = 1, h_2 = x$ so that $m_1 = \frac{x}{2}, m_2 = \frac{1}{2}$. We define the one-parameter family of operators $U(t)$ via
\begin{gather*}
    U(t) := \begin{bmatrix}
        (h_1 m_1, e^{i \lambda t}) & (h_1 m_2, e^{i \lambda t}) \\
        (h_2 m_1, e^{i \lambda t}) & (h_2 m_2, e^{i \lambda t})
    \end{bmatrix} = \begin{bmatrix}
        \cos(t \lambda) & \frac{\sin(t \lambda)}{\lambda} \\
        -\lambda \sin(t \lambda) & \cos(t \lambda)
    \end{bmatrix},
\end{gather*}
where we treat each polynomial $h_j m_k$ as a polynomial in $i \lambda$.
Let $A$ be the matrix defined via
\begin{gather*}
    A := \begin{bmatrix}
        (m_1 h_1, i \lambda) & (m_2 h_1, i \lambda) \\
        (m_1 h_2, i \lambda) & (m_2 h_2, i \lambda)
    \end{bmatrix} = \begin{bmatrix}
        0 & 1 \\
        - \lambda^2 & 0
    \end{bmatrix},
\end{gather*}
where again we treat each $m_j h_k$ as a polynomial in $i \lambda$. We have
\begin{gather}
    A = \begin{bmatrix}
        \frac{1}{2 i \lambda} & -\frac{1}{2 i \lambda} \\
        \frac{1}{2} & \frac{1}{2}
    \end{bmatrix} \begin{bmatrix}
        i \lambda & 0 \\
        0 & -i \lambda
    \end{bmatrix}
    \begin{bmatrix}
        i \lambda & 1 \\
        -i \lambda & 1
    \end{bmatrix}. \label{eqn_diagonalize_A}
\end{gather}
We then have that $U(t) = \exp(t A)$, and thus the eigenvalues of $U(t)$ are $e^{i \lambda t}, e^{-i \lambda t}$. 

Let $u(x, t)$ be a solution to the wave equation with $u(x, 0) = f(x)$ and $u_t(x, 0) = g(x)$. Suppose $f, g$ are such that we make take their Fourier transforms (in the sense of functions, rather than distributions). We then have
\begin{gather*}
    \begin{bmatrix}
        \mathcal{F}(u(\cdot, t))(\lambda) \\
        \mathcal{F}( u_t(\cdot, t))(\lambda)
    \end{bmatrix} = U(t) \begin{bmatrix}
        \hat{f}(\lambda) \\
        \hat{g}(\lambda)
    \end{bmatrix},
\end{gather*}
i.e. $U(t)$ is the time $t$ wave propagator with respect to the standard initial conditions.

Let $B$ be the last matrix appearing in \eqref{eqn_diagonalize_A}. Define
\begin{gather*}
    Q := \frac{1}{2}B^* B = \begin{bmatrix}
        -\lambda^2 & 0 \\
        0 & 1
    \end{bmatrix} .
\end{gather*}
Notice that, by construction, $U(t)^* Q U(t) = Q$ for all $t$. Therefore
\begin{align*}
    \begin{bmatrix}
        \overline{\mathcal{F}(u(\cdot, t))}(\lambda) & \overline{\mathcal{F}(u_t(\cdot, t))}(\lambda)
    \end{bmatrix} Q \begin{bmatrix}
        \mathcal{F}(u(\cdot, t))(\lambda) \\
        \mathcal{F}(u_t(\cdot, t))(\lambda)
    \end{bmatrix} &= \begin{bmatrix}
        \overline{\hat{f}}(\lambda) & \overline{\hat{g}}(\lambda)
    \end{bmatrix} U(t)^* Q U(t) \begin{bmatrix}
        \hat{f}(\lambda) \\
        \hat{g}(\lambda)
    \end{bmatrix} \\
    &= \begin{bmatrix}
        \overline{\hat{f}}(\lambda) & \overline{\hat{g}}(\lambda)
    \end{bmatrix} Q \begin{bmatrix}
        \hat{f}(\lambda) \\
        \hat{g}(\lambda)
    \end{bmatrix}
\end{align*}
which is independent of $t$. We thus have
\begin{gather*}
    \int_{-\infty}^\infty -|\mathcal{F}(u(\cdot, t))(\lambda)|^2 \lambda^2 + |\mathcal{F}(u_t(\cdot, t))(\lambda)|^2 d\lambda = \int_{-\infty}^\infty -|\hat{f}(\lambda)|^2 \lambda^2 + |\hat{g}(\lambda)|^2 d\lambda
\end{gather*}
Therefore, by the Plancherel formula, we get
\begin{align*}
    E(t) :&= -\int_{-\infty}^\infty \overline{u(x, t)} \cdot \frac{\partial^2}{\partial x^2} u(x, t) dx + \int_{-\infty}^\infty |u_t(x, t)|^2 dx \\
    &= \|u_x(x, t) \|^2_{L^2(\mathbb{R}, dx)} + \| u_t(x, t) \|^2_{L^2(\mathbb{R}, dx)} = \| f'(x) \|^2_{L^2(\mathbb{R}, dx)} + \|g(x) \|^2_{L^2(\mathbb{R}, dx)},
\end{align*}
so $E(t)$ is constant in time. In fact, if we write $u(x, t) = w_\infty(x-t) + w_{-\infty}(x+t)$, then 
\begin{gather*}
    E(t) = 2 \int_{-\infty}^\infty | w_\infty'(k) |^2 dk + 2 \int_{-\infty}^\infty | w'_{-\infty}(k) |^2 dk.
\end{gather*}

\section{The wave equation on regular trees}
\subsection{Harmonic analysis on regular trees} \label{sec_harmonic_analysis_trees}

Let $\mathcal{T}$ denote the $(q+1)$-regular tree. We shall often speak about functions on $\mathcal{T}$, by which we mean functions on the vertices of $\mathcal{T}$. Let $C(\mathcal{T}; \mathbb{C})$ denote the space of such functions. Let $o$ denote some fixed vertex of $\mathcal{T}$; we shall take $o$ to be our ``origin.'' Let $A$ denote the \textit{adjacency operator} on $\mathcal{T}$, i.e. given a $f \in C(\mathcal{T}; \mathbb{C})$, we have
\begin{gather*}
    [A f](v) = \sum_{w : w \sim v} f(w),
\end{gather*}
where $w \sim v$ means that the vertices $w$ and $v$ are adjacent. Let $G = \textnormal{Aut}(\mathcal{T})$, and let $K < G$ be the stabilizer of $o$. We may naturally identify the vertices of $\mathcal{T}$ with the coset space $G/K$, and thus a function on the vertices of $\mathcal{T}$ is the same thing as a right $K$-invariant function on $G$. We say that a function on the vertices of $\mathcal{T}$ is \textit{radial} if its value only depends on the distance from $o$; this is equivalent to the function being $K$-invariant. We may in turn think of a radial function as a bi-$K$-invariant function on $G$. The collection of compactly supported bi-$K$-invariant functions on $G$ forms an algebra under convolution known as the \textit{spherical Hecke algebra} $H(G, K)$. By treating functions on vertices of $\mathcal{T}$ as right $K$-invariant functions on $G$, we see that $H(G, K)$ naturally acts on functions on vertices via convolution; we shall denote this as $f * g$ with $f \in H(G, K)$ and $g$ a function on $\mathcal{T}$. We can identify the adjacency operator with the radial function which is 1 on the vertices at distance 1 from $o$ and 0 elsewhere. A vector space basis for $H(G, K)$ consists of radial functions which are supported on spheres of a given radius centered at $o$. On the other hand, as an algebra $H(G, K)$ is freely generated by the adjacency operator (and in particular $H(G, K)$ is a commutative algebra). 

For every $z \in \mathbb{C}$, there is a unique radial function $\phi_z: \mathcal{T} \to \mathbb{C}$ such that $\phi_z(o) = 1$ and $\frac{A}{2 \sqrt{q}} \phi_z = z \phi_z$. This is called the \textit{spherical function} (with respect to $o$) with eigenvalue $z$. Any element $f \in H(G, K)$ can be expressed as $f = g(\frac{A}{2 \sqrt{q}})$ where $g$ is some polynomial. Then $f * \phi_z = g(z) \phi_z$. We can define a map:
\begin{gather*}
    \textnormal{Sph} : H(G, K) \to \mathbb{C}[z]
\end{gather*}
by sending $f = g(\frac{A}{2 \sqrt{q}})$ to $g(z)$. This is an isomorphism of algebras. We shall refer to this map as the \textit{spherical transform}.

There's another natural algebra isomorphism of $H(G, K)$ called the \textit{Satake isomorphism}:
\begin{gather*}
    \textnormal{Sat}: H(G, K) \to \mathbb{C}[x, x^{-1}]^W.
\end{gather*}
One could simply define this map by sending $\frac{A}{2 \sqrt{q}}$ to $\frac{x + x^{-1}}{2}$ and extending to a homomorphism. However, a more geometric definition may be given in terms of integration over horocycles which we now introduce.

Suppose $\gamma_1, \gamma_2: \mathbb{N} \to \mathcal{T}$ are two infinite non-backtracking paths in $\mathcal{T}$. We say that $\gamma_1$ and $\gamma_2$ are equivalent if 
\begin{gather*}
    \lim_{n \to \infty} d(\gamma_1(n), \gamma_2(n)) < \infty.
\end{gather*}
This is equivalent to the traces of $\gamma_1$ and $\gamma_2$ eventually being identical. It is clear that this defines an equivalence relation on infinite non-backtracking paths. Let $\Omega$ denote the set of equivalence classes. Given $\omega \in \Omega$ and a vertex $v \in \mathcal{T}$, there is a unique non-backtracking path starting at $v$ which is in the equivalence class of $\omega$; we denote this path as $\gamma_{v, \omega}$. Suppose $o \to v_1 \to \dots \to v_n$ is some non-backtracking path starting from $o$. Let $\Omega(v_n)$ denote all extensions of this path to an infinite non-backtracking path; such sets are called \textit{cylindrical}. We can define a natural topology on $\Omega$ by taking all $\Omega(v)$ (with $v$ arbitrary) as a basis of open sets. We may define a natural probability measure $\nu$ on $\Omega$ by assigning $\Omega(v_n)$ the measure $\frac{1}{(q+1)q^{n-1}}$. Then $\nu$ is the unique measure on $\Omega$ invariant under $K$. We shall refer to $\nu$ as the \textit{harmonic measure} (with respect to $o$).

Given $\omega \in \Omega$, we define the associated \textit{Busemann function} $h_\omega: \mathcal{T} \to \mathbb{Z}$ via
\begin{gather} \label{eqn_busemann}
    h_\omega(v) := \lim_{n \to \infty} n - d(\gamma_{o, \omega}(n), v).
\end{gather}
Notice that $h_\omega(o) = 0$ for all $\omega$. The level sets of $h_\omega$ are called the \textit{horocycles} with respect to $\omega$.

We can now define the Satake isomorphism. Let $f \in H(G, K)$, which we think of as a compactly supported radial function. Choose any $\omega \in \Omega$; the choice of $\omega$ will ultimately not matter. Then $\textnormal{Sat}(f) \in \mathbb{C}[x, x^{-1}]^W$ is defined via
\begin{gather} \label{eqn_satake_formula}
    [x^k]\big( \textnormal{Sat}(f) \big) := q^{\frac{k}{2}} \sum_{v \in h_{\omega}^{-1}(k)} f(v).
\end{gather}
For example the function which is 1 on the sphere of radius 1 and zero elsewhere (i.e. the adjacency operator $A$) maps to 
\begin{gather} \label{eqn_image_satake}
    q^{1/2} \cdot 1 \cdot x + q^{-1/2} \cdot q \cdot x^{-1} = \sqrt{q}(x + x^{-1}).
\end{gather}
It is quite remarkable that this process always gives an element in $\mathbb{C}[x, x^{-1}]^W$ and that it is an algebra isomorphism.

We can think of each element in $\mathbb{C}[x, x^{-1}]$ as a compactly supported function on $\mathbb{Z}$ or as a meromorphic function on $\mathbb{C}^\times$. However, when thinking of such elements in the latter manner, we shall opt to instead use the variable $\xi$. Then $\mathbb{C}[\xi, \xi^{-1}]^W$ is those functions invariant under $\xi \mapsto \xi^{-1}$. We can think of these as functions on the variety
\begin{gather*}
    \Theta := \{(\xi, \xi^{-1}) \in (\mathbb{C}^\times)^2\}/\Big((\xi, \xi^{-1}) \sim (\xi^{-1}, \xi)\Big)
\end{gather*}
We have an algebraic map from $\Theta$ to $\mathbb{C}$ given by $(\xi, \xi^{-1}) \mapsto \frac{\xi + \xi^{-1}}{2}$. This induces an isomorphism of coordinates rings $\mathbb{C}[z] \simeq \mathbb{C}[\xi, \xi^{-1}]^W$. Previously $z \in \mathbb{C}$ was parametrizing possible eigenvalues for $\frac{A}{2 \sqrt{q}}$, but writing $z = \frac{\xi + \xi^{-1}}{2}$, we may instead parametrize by pairs $(\xi, \xi^{-1})$, modulo switching the coordinates, i.e. points in $\Theta$. This latter parametrization is known as the \textit{Satake parameters}. If we compose the Satake map with the aforementioned isomorphism from $\mathbb{C}[\xi, \xi^{-1}]^W$ to $\mathbb{C}[z]$, we exactly obtain the spherical transformation. This is immediate by just checking where $\frac{A}{2 \sqrt{q}}$ gets mapped.

The spherical transform can be extended to an isometry:
\begin{gather*}
    L^2( G / \hspace{-1.5mm} / K) \to L^2 (\mathbb{C}, d \mu_{\textnormal{Pl}}),
\end{gather*}
where $L^2 (G / \hspace{-1.5mm} / K)$ is radial functions (not necessarily compactly supported) which are in $L^2(\mathcal{T})$ with respect to the usual $L^2$-norm, and where $d \mu_{\textnormal{Pl}}$ is a certain measure called the Plancherel measure (also known as the Kesten--McKay law) which is supported on $[-1, 1]$.

It is easiest to express the Plancherel measure using the Satake parameters rather than the eigenvalue. The $c$-function is the function on $\mathbb{C}^\times$ defined by
\begin{gather}
    c(\xi) = \frac{\xi - q^{-1} \xi^{-1}}{\xi - \xi^{-1}}. \label{eqn_c_function}
\end{gather}
In terms of the Satake parameters, the Plancherel measure is supported on the sublocus of $(\xi, \xi^{-1}) \in \Theta$ where $|\xi| = 1$. It is often easier to work with the (branched) double cover of $\Theta$
\begin{gather*}
    \tilde{\Theta} := \{(\xi, \xi^{-1}) \in (\mathbb{C}^\times)^2\} \simeq \mathbb{C}^\times.
\end{gather*}
Then the Plancherel measure is supported on $S^1$. If we parametrize this locus as $\xi = e^{i \theta}$, then the Plancherel measure can be expressed as
\begin{gather*}
    d \mu_{\textnormal{Pl}} = \frac{(1 + q^{-1})}{4 \pi} \frac{1}{c(\xi) c(\xi^{-1})} d \theta.
\end{gather*}
If we change to the eigenvalue parametrization, the Plancherel measure is given by the formula in Example \ref{example_kesten_mckay}.

In terms of the Satake parameters, the spherical function, thought of as a function on $\mathbb{N}$ (i.e. its value on the sphere of radius $t$), is
\begin{gather} \label{eqn_spherical_function_hc_expansion}
    \phi_{(\xi, \xi^{-1})}(t) = \frac{1}{(1 + q^{-1})} q^{-\frac{t}{2}} \big(\xi^t c(\xi) + \xi^{-t} c(\xi^{-1}) \big).
\end{gather}
In terms of the eigenvalue $z$ it is
\begin{gather} \label{eqn_spherical_function_chebyshev}
    \phi_z(t) = q^{-\frac{t}{2}} \Big( \frac{2}{q+1} T_t(z) + \frac{q-1}{q+1} U_t(z) \Big).
\end{gather}

\begin{remark} \label{remark_f_n_one}
The parenthetical part of \eqref{eqn_spherical_function_hc_expansion} is exactly of the form as in Proposition \ref{prop_spectral}, i.e. it solves the spectral wave equation. On the other hand the parenthetical part of \eqref{eqn_spherical_function_chebyshev} is in fact exactly equal to $\frac{1}{(1+q^{-1})} F_t$, the Chebyshev polynomial corresponding to $h = \xi - q^{-1} \xi^{-1}$ discussed in Example \ref{example_kesten_mckay}:
\begin{align*}
    \frac{2}{q+1} T_t + \frac{q-1}{q+1} U_t &= \frac{2}{q+1} \Big( \frac{U_t - U_{t-2}}{2} \Big) + \frac{q-1}{q+1} U_t (z) = \frac{q}{q+1} U_t - \frac{1}{q+1} U_{t-2} \\
    &= \frac{1}{(1 + q^{-1})}(U_t - q^{-1} U_{t-2}) = \frac{1}{(1+q^{-1})} F_t.
\end{align*}

\end{remark}

We can invert the spherical transform, and thus also the Satake transform. If $f \in H(G, K)$, and $v \in \mathcal{T}$ with $d(o, v) = \ell$, then
\begin{gather} \label{eqn_inverse_spherical}
    f(v) = \int_{S^1} \textnormal{Sat}(f)(\xi) \cdot  \phi_{(\xi, \xi^{-1})}(\ell)  \ d \mu_{\textnormal{Pl}} (\xi).
\end{gather}

We now discuss the Poisson transform on $\mathcal{T}$. The main reference for this is Chapter II of \cite{figa_talamanca_nebbia}. Let $\mathcal{K}(\Omega)$ denote the space of continuous functions on $\Omega$ which take only a finite number of values. Let $\mathcal{K}'(\Omega)$ denote the dual space of $\mathcal{K}(\Omega)$; we refer to elements in $\mathcal{K}'(\Omega)$ as finitely additive measures. 

\begin{theorem}[Theorem 1.2 in Chapter II of \cite{figa_talamanca_nebbia}]
Suppose $f \in C(\mathcal{T}; \mathbb{C})$ such that $\frac{A}{2 \sqrt{q}} f = \Big(\frac{q^{is} + q^{-is}}{2} \Big) f$ with $q^{is} \neq \pm \frac{1}{\sqrt{q}}$. Then there exists a unique $m \in \mathcal{K}'(\Omega)$ such that
\begin{gather*}
    f(v) = \mathcal{P}_{q^{is}}(m) := \int_{\Omega} q^{(\frac{1}{2} + i s) h_{\omega}(v)} dm (\omega).
\end{gather*} 
\end{theorem}
We remark that for fixed $x$, we have $h_\omega(x) \in \mathcal{K}(\Omega)$, so we may indeed pair the integrand with $m$. If we just take $m$ to correspond to integration against the function on $\Omega$ which is identically 1, then the corresponding eigenfunction is exactly the spherical function. We refer to $\mathcal{P}_{q^{is}}(m)$ as the \textit{Poisson transform} of $m$ (with respect to the spectral parameter $q^{is}$). 

Note that we naturally have that $\mathcal{K}(\Omega) \subset \mathcal{K}'(\Omega)$. As discussed in Chapter II Section 3 of \cite{figa_talamanca_nebbia}, so long as $q^{is} \neq \pm \frac{1}{\sqrt{q}}$, then we may define a map $I_{q^{is}}: \mathcal{K}(\Omega) \to \mathcal{K}(\Omega)$ such that for all $g \in \mathcal{K}(\Omega)$, 
\begin{gather*}
    \mathcal{P}_{q^{is}}(I_{q^{is}} g \ d \nu)(v) = \mathcal{P}_{q^{-is}}(g \ d \nu)(v). 
\end{gather*}
Then $I_{q^{is}}$ is called the \textit{intertwining operator} associated to $q^{is}$. Letting $s$ vary, we may in fact think of $q^{is}$ as a meromorphic family of operators with poles at $q^{is} = \pm \frac{1}{\sqrt{q}}$, and which extends to a unitary operator on $L^2(\Omega, \nu)$ for $|q^{is}| = 1$. We shall see in Section \ref{sec_scattering_theory} that $I_{q^{is}}$ is essentially the Fourier transform of the Lax--Phillips scattering operator.

Note that $\nu$ is only quasi-invariant under $G$; its Radon--Nikodym derivative $\frac{d \nu(g^{-1} \omega)}{d \nu(\omega)}$ is given by $P(g.o, \omega)$ where we define
\begin{gather*}
    P(x, \omega) := q^{h_\omega(x)}.
\end{gather*}
We let $\phi_{(\xi, \xi^{-1})}^x$ denote the spherical function centered at $x$, which may be defined as $\phi_{(\xi, \xi^{-1})}^x(v) = \phi_{(\xi, \xi^{-1})}(g^{-1} .v)$ where $g \in G$ is any element such that $g.o = x$. 
\begin{proposition} \label{prop_spherical_function_new_center}
    Let $x \in \mathcal{T}$ be a vertex. The spherical function centered at $x$ with Satake parameters $(q^{is}, q^{-is})$ may be expressed as
    \begin{gather}
        \phi_{(q^{is}, q^{-is})}^x(v) = \int_{\Omega} q^{(\frac{1}{2} + i s)h_\omega(v)} q^{(\frac{1}{2} - is) h_\omega(x)} \ d \nu(\omega) = \mathcal{P}_{q^{is}}( q^{(\frac{1}{2} - is) h_\omega(x)} \ d \nu). \label{eqn_poisson_other_spherical}
    \end{gather}
\end{proposition}
\begin{proof}
    We have
    \begin{gather*}
        \phi^x_{(q^{is}, q^{-is})}(v) = \int_{\Omega} q^{(\frac{1}{2} + is) h^x_\omega(v)} d \nu^x(\omega),
    \end{gather*}
    where $h^x_\omega(v)$ is the Busemann function shifted so that $h^x_\omega(x) = 0$ for all $\omega$, and $d \nu^x$ is the harmonic measure from the perspective of $x$. Thus $h^x_\omega(v) = h_\omega(v) - h_\omega(x)$. Furthermore, the Radon--Nikodym derivative $\frac{d \nu^x(\omega)}{d \nu(\omega)}$ is equal to $q^{h_\omega(x)}$. From here the result follows immediately.
\end{proof}

We shall also make use of the \textit{Helgason transform}. This is essentially the Fourier transform of the horocyclic Radon transform. More specifically given $f \in C_c(\mathcal{T})$ we have $\textnormal{Hel}(f) \in C(S^1 \times \Omega)$ defined via:
\begin{gather*}
    [\textnormal{Hel}(f)](q^{is}, \omega) := \sum_{x \in \mathcal{T}} f(x) q^{(\frac{1}{2} + i s)h_\omega(x)}. 
\end{gather*}

\begin{proposition}[see e.g. \cite{cowling_setti}] \label{prop_helgason_transform}
    The Helgason transform extends to an isometric embedding of $L^2(\mathcal{T})$ into $L^2(S^1 \times \Omega, d \mu_{\textnormal{Pl}} \times d \nu)$. Its image is those functions $F(q^{is}, \omega)$ satisfying
        \begin{gather}
            \int_\Omega q^{(\frac{1}{2} - is) h_\omega(x)} F(q^{is}, \omega) d \nu(\omega) = \int_{\Omega} q^{(\frac{1}{2} + i s) h_\omega(x)} F(q^{-is}, \omega) d \nu(\omega), \label{eqn_helgason_symmetry}
        \end{gather}
        for all $x \in \mathcal{T}$ and almost all $q^{is} \in S^1$. 
    \end{proposition}

\subsection{Wave equation on regular trees} \label{sec_wave_eqn_tree}
We say that $u: \mathcal{T} \times \mathbb{Z} \to \mathbb{C}$ satisfies the wave equation if
\begin{gather} \label{wave_eqn_tree}
    \frac{A_v}{2 \sqrt{q}} u(v, t) = \frac{u(v, t+1) + u(v, t-1)}{2}.
\end{gather}
To be clear, on the LHS, we fix $t$ and apply $\frac{A}{2 \sqrt{q}}$ to the resulting function on $\mathcal{T}$. Extending slightly our previous notation, we define
\begin{gather*}
    u^{[y]}(v, \cdot) := \sum_{n \in \mathbb{Z}} u(v, n) y^n.
\end{gather*}
Then the RHS of \eqref{wave_eqn_tree} can be re-expressed as
\begin{gather*}
    [y^t]\Big(u^{[y]}(v, \cdot) \cdot \big(\frac{y + y^{-1}}{2}\big)\Big)
\end{gather*}
Notice that \eqref{eqn_image_satake} shows that $\frac{A}{2 \sqrt{q}}$ maps to $\frac{x + x^{-1}}{2}$ under the Satake isomorphism. Because $H(G, K)$ is commutative and generated by $A$, we could equivalently express \eqref{wave_eqn_tree} as the condition that for every $B \in H(G, K)$ we have
\begin{gather*}
    B_v u(v, t) = [y^t] \Big( u^{[y]}(v, \cdot) \cdot \textnormal{Sat}(B)(y) \Big).
\end{gather*}

\begin{theorem} \label{thm_solve_wave_tree}
    Let $h_1, h_2 \in \mathbb{C}[x, x^{-1}]$ be a free basis as a module over $\mathbb{C}[x, x^{-1}]^W$. Let $m_1, m_2$ be the corresponding dual basis. Let $\textnormal{Ch}_t^{m_i}(z)$ denote the corresponding Chebyshev polynomials (see \eqref{eqn_chebyshev}). Let $g_1, g_2 : \mathcal{T} \to \mathbb{C}$. Then the unique solution $u(v, t)$ to the wave equation \eqref{wave_eqn_tree} satisfying
    \begin{align}
        [y^0](u^{[y]}(v, \cdot) \cdot h_1(y^{-1})) &= g_1(v) \label{eqn_tree_initial} \\
        [y^0](u^{[y]}(v, \cdot) \cdot h_2(y^{-1}) &= g_2(v) \notag
    \end{align}
    is given by
    \begin{gather} \label{eqn_tree_soln}
        u(v, t) = \Big( \textnormal{Ch}^{m_1}_t \Big( \frac{A}{2 \sqrt{q}} \Big)g_1 + \textnormal{Ch}^{m_2}_t \Big( \frac{A}{2 \sqrt{q}} \Big)g_2 \Big)(v).
    \end{gather}
\end{theorem}

\begin{proof}
    Suppose first that $g_1, g_2$ are compactly supported radial functions. It is clear that any solution to \eqref{wave_eqn_tree} with such initial conditions must have the property that for every fixed $t$, the function $u(v, t)$ is radial and compactly supported. We may thus apply the Satake transform to express the solution as a function $[x^n]\textnormal{Sat}(u (\cdot, t))$ which maps $\mathbb{Z}^2$ to $\mathbb{C}$. Then \eqref{wave_eqn_tree} turns into \eqref{eqn_discrete_wave_intro}, and \eqref{eqn_tree_initial} and \eqref{eqn_tree_soln} turn into \eqref{eqn_flat_initial} and \eqref{eqn_soln_initial} respectively. In this case existence and uniqueness follow immediately from Theorem \ref{thm_soln}.

    Now for the general case. It is clear that $G$ preserves solutions to \eqref{wave_eqn_tree}. More specifically, given $g \in G$, suppose $g.o = w$. Suppose $u(v, t)$ solves the wave equation with compactly supported radial initial data $g_1, g_2$ satisfying \eqref{eqn_tree_initial}. Then $j(v, t) := u(g^{-1}.v, t)$ also solves the wave equation with initial conditions $g_1(g^{-1}.v)$ and $g_2(g^{-1}.v)$. This solution is in turn radial and compactly supported for each $t$, but now centered around $w$. We may in turn apply the Satake transform, but with $o$ in \eqref{eqn_satake_formula} replaced by $w$ (and thus also the appropriate change in \eqref{eqn_busemann}). We may in turn appeal to Theorem \ref{thm_soln} to conclude existence and uniqueness of solution, and the validity of \eqref{eqn_tree_soln}.

    Finally, it is clear that we may add together solutions to \eqref{wave_eqn_tree} and still obtain a solution. We may rewrite $u(v, t)$ as
    \begin{gather*}
        u(v, t) = \sum_{w \in \mathcal{T}} \Big(\textnormal{Ch}_t^{m_1} \Big( \frac{A}{2 \sqrt{q}} \Big) (g_1(w) \delta_w) \Big) (v) + \Big(\textnormal{Ch}_t^{m_2} \Big( \frac{A}{2 \sqrt{q}} \Big) (g_2(w) \delta_w) \Big) (v).
    \end{gather*}
    Though we are adding together infinitely many terms, for any fixed vertex $v$ and time $t$, there are only finitely many terms which are non-zero at $v$. From the above discussion it is clear that $u(v, t)$ solves the wave equation and has the desired initial conditions. On the other hand, uniqueness of solution is clear since any two solutions with the same initial conditions would differ by a solution with zero initial conditions. However the zero function is clearly radial and compactly supported so we can again appeal to Theorem \ref{thm_soln} to get uniqueness of solution.
\end{proof}

Suppose we take as our initial conditions $h_1 = 1, h_2 = \frac{x - x^{-1}}{2}$. Then \eqref{eqn_tree_soln} turns into
\begin{gather}
    u(v, t) = T_t \Big( \frac{A}{2 \sqrt{q}} \Big) g_1(v) + U_{t-1} \Big( \frac{A}{2 \sqrt{q}} \Big) g_2(v). \label{eqn_tree_std_initial}
\end{gather}
This is the form of the solution to the wave equation described in Brooks--Lindenstrauss \cite{brooks_lindenstrauss}. Finite speed of propagation is also immediate from \eqref{eqn_tree_std_initial}, namely if $u(v, t)$ is a solution to the wave equation with standard initial conditions $f, g$, then if $f$ is supported on $d(o, v) \leq r$ and $g$ is supported on $d(o, v) \leq r+1$, then $u(n, t)$ is supported on $d(o, v) \leq r + |t|$ for all $t$.

We now consider conservation of energy.
\begin{proposition} \label{prop_tree_energy}
    Suppose $u(v, t)$ is the solution to the wave equation. Then
    \begin{gather}
        E(t) := \sum_{v} \overline{u(v, t)} \cdot \Big(1 - \frac{A^2}{4 q} \Big)u(v, t) + \sum_{v} \Big \| \frac{u(v, t+1) - u(v, t-1)}{2} \Big \|^2 \label{eqn_energy_tree}
    \end{gather}
    is independent of $t$.
\end{proposition}
\begin{proof}
    This essentially follows from the analysis preceding Proposition \ref{prop_energy_flat}. Suppose $u(v, t)$ is a solution to wave equation, and let $j(v, t) := \frac{u(v, t+1) - u(v, t-1)}{2}$. We then have
    \begin{gather} \label{eqn_tree_one_time_step}
        \begin{bmatrix}
            u(v, t+1) \\
            j(v, t+1)
        \end{bmatrix} = \begin{bmatrix}
            \frac{A}{2 \sqrt{q}} & I \\
            \frac{A^2}{4 q} - I & \frac{A}{2 \sqrt{q}}
        \end{bmatrix} \begin{bmatrix}
            u(v, t) \\
            j(v, t)
        \end{bmatrix}.
    \end{gather}
    This matrix is simply the matrix $\mathcal{U}$ from \eqref{eqn_one_time_step_prop_U} with $z$ replaced by $\frac{A}{2 \sqrt{q}}$. The argument from Section \ref{sec_energy}, immediately shows that the expression in \eqref{eqn_energy_tree} is preserved in time.
\end{proof}

We remark that the spectrum of $\frac{A}{2 \sqrt{q}}$ acting on $L^2(\mathcal{T})$ is $[-1, 1]$. Therefore the spectrum of $1 - \frac{A^2}{4 q}$ is $[0, 1]$. If $g \in L^2(\mathcal{T})$ is compactly supported, then $\langle \big(1 - \frac{A^2}{4 q} \big) g, g \rangle > 0$. Given $(g_1, g_2)$ and $(f_1, f_2)$, all of which are compactly supported functions on $\mathcal{T}$, we define the positive definite energy form
\begin{gather} \label{eqn_energy_tree_formula}
    [(g_1, g_2), (f_1, f_2)] := \langle \big(1 - \frac{A^2}{4 q} \big) g_1, f_1 \rangle + \langle g_2, f_2 \rangle.
\end{gather}
Let $\mathcal{H}$ be the Hilbert space completion with respect to this energy form of the space of pairs of compactly supported functions. Then we can interpret the proof of Proposition \ref{prop_tree_energy} as telling us that the operator in \eqref{eqn_tree_one_time_step} acts unitarily on $\mathcal{H}$. 

We now wish to give more concrete descriptions of $T_t(\frac{A}{2 \sqrt{q}})$ and $U_t(\frac{A}{2 \sqrt{q}})$. We shall also be interested in $F_t(\frac{A}{2 \sqrt{q}})$.

\begin{proposition} \label{prop_reg_tree_chebs}
    For $t \geq 0$ we have that
    \begin{align}
        U_t \Big(\frac{A}{2 \sqrt{q}} \Big) \delta_o(v) &= \begin{cases}
            q^{-\frac{t}{2}} & \textnormal{if $t - d(o, v) \geq 0$ and even} \\
            0 & \textnormal{otherwise}.
        \end{cases} \label{eqn_cheb_2_tree} \\
        T_t \Big(\frac{A}{2 \sqrt{q}} \Big) \delta_o(v) &= \begin{cases}
            \frac{1-q}{2} q^{-\frac{t}{2}} & \textnormal{if $t - d(o, v) > 0$ and even} \\
            \frac{1}{2} q^{-\frac{t}{2}} & \textnormal{if $d(o, v) = t$} \\
            0 & \textnormal{otherwise}.
        \end{cases} \label{eqn_cheb_1_tree} \\
        F_t \Big(\frac{A}{2 \sqrt{q}} \Big) \delta_o(v) &= \begin{cases}
            q^{-\frac{t}{2}} & \textnormal{if $d(o, v) = t$} \\
            0 & \textnormal{otherwise}.
        \end{cases} \label{eqn_spherical_soln_tree}
    \end{align}
\end{proposition}

\begin{proof}
    It will ultimately suffice to prove the formula for $U_t$. Since $T_t = \frac{1}{2}(U_t - U_{t-2})$ and $F_t = U_t - q^{-1} U_{t-2}$ the other formulas will follow. 

    We use \eqref{eqn_inverse_spherical}. To compute $U_t(\frac{A}{2 \sqrt{q}}) \delta_o(v)$ on a sphere of radius $\ell$ centered at $o$, we must compute
    \begin{gather*}
        \frac{(1+q^{-1})}{4 \pi} \int_{S^1} \frac{\xi^{t+1} - \xi^{-t-1}}{\xi - \xi^{-1}} \frac{1}{(1 + q^{-1})} q^{-\ell/2} \big( \xi^\ell c(\xi) + \xi^{-\ell} c(\xi^{-1}) \big) \frac{1}{c(\xi) c(\xi^{-1})} d \theta.
    \end{gather*}
    When we expand out and simplify, we get $\frac{q^{-\ell/2}}{4 \pi}$ times the integral over $S^1$ of
    \begin{gather*}
        -\frac{\xi^{t+\ell+2}}{1 - q^{-1} \xi^2} + \frac{\xi^{t-\ell}}{1 - q^{-1} \xi^{-2}} + \frac{\xi^{-t + \ell}}{1 - q^{-1} \xi^{2}} - \frac{\xi^{-t-\ell-2}}{1 - q^{-1} \xi^{-2}}.
    \end{gather*}
    Because $q^{-1} < 1$, we can compute the integral by expanding the above expression as Laurent series in $\xi$ centered at 0 and reading off the constant term. We always have $t +\ell + 2 > 0$, so the first and fourth terms have zero constant term. If $t - \ell$ is negative or not even, then the second and third terms have zero constant term. Otherwise if $t - \ell \geq 0$ and even, then the second and third terms each give $2 \pi (q^{-1})^{(t - \ell)/2}$. Multiplying this by $\frac{q^{-\ell/2}}{4 \pi}$ gives the result.
\end{proof}

The explicit formulas given in Proposition \ref{prop_reg_tree_chebs} can be found independently in \cite{anker_martinot_pedon_setti} and in \cite{brooks_lindenstrauss, brooks_lindenstrauss_graphs}.

\begin{remark}
    From a more representation theoretic perspective, we have that $U_t(\frac{x + x^{-1}}{2}) = x^t + x^{t-2} + \dots + x^{2-t} + x^{-t}$ gives the Schur polynomial associated to the irreducible highest weight module $V_{(t, 0)}$ of $\mathfrak{sl}(2)$ with highest weight $(t, 0)$ (viewed as a partition). For a general split semi-simple adjoint algebraic group over a non-archimedean local field with residue field of order $q$, the preimage of the Schur polynomial $s_\lambda$ under the Satake isomorphism (centered at some special vertex $o$) associated to a dominant coweight $\lambda$ is supported on special vertices $v$ whose Weyl chamber-valued distance from $o$ is $\preceq \lambda$ (where $\mu \preceq \lambda$ means $\lambda - \mu$ can be expressed as an positive integral sum of fundamental coroots). The value on the ``spherical shell'' of radius $\mu$ is given by $q^{-\langle \lambda, \rho \rangle} P_{\mu, \lambda}(q)$ where $P_{\mu, \lambda}(q)$ is the Kazhdan-Lusztig polynomial (viewing $\mu, \lambda$ as elements in the extended affine Weyl group) and $\rho$ is the half sum of positive roots. The $q$-analogue of the Kostant multiplicity theorem (which is due to Kato \cite{kato}) tells us that if $\textnormal{dim}(V_\lambda(\mu)) = 1$, then $P_{\mu, \lambda}(q) = 1$. Since all irreducible representations of $\mathfrak{sl}(2)$ have $\textnormal{dim}(V_{\lambda}(\mu)) = 1$ if $\mu \preceq \lambda$, we immediately conclude that the preimage of $x^t + x^{t-2} + \dots + x^{2-t} + x^{-t}$ under the Satake transform is exactly given by the formula in \eqref{eqn_cheb_2_tree} (note that $\rho = (\frac{1}{2}, 0)$). See \cite{gross}. We shall return to this point in follow-up work with Anker, Rémy, and Trojan on the multitemporal wave equation on affine buildings where we shall see that the Schur polynomials will provide, in some sense, the ``most important'' fundamental solution to the ``flat'' multitemporal wave equation \eqref{intro_higher_rank_wave}, and consequently the inverse of the Schur polynomials under the Satake isomorphism will give the ``most important'' fundamental solution to the multitemporal wave equation on the building \eqref{intro_higher_rank_building}.
\end{remark} 

\begin{remark}
    We continue Remark \ref{remark_f_n_one}. The discussion there foreshadowed that the spherical functions provide a solution to the wave equation on the tree, after we remove the factor $q^{-t/2}$. This is exactly the solution corresponding to $F_t(\frac{A}{2 \sqrt{q}})$ (or possibly a rescaled version of it). Using \eqref{eqn_spherical_soln_tree}, we see that the functions $F_t(\frac{A}{2 \sqrt{q}})$ for different $t \geq 0$ provide an orthogonal basis for $L^2( G / \hspace{-1.5mm} / K)$. On the other hand, their image under the spherical transform is exactly $F_t(z)$ which, via Example \ref{example_kesten_mckay}, is an orthogonal basis for the $L^2$-space with respect to the Kesten--McKay law, i.e. the Plancherel measure expressed in $z$-coordinates. This is simply a reformulation of the aforementioned fact that the spherical transform is an isometry $L^2( G / \hspace{-1.5mm} / K) \to L^2(\mathbb{C}, d \mu_{\textnormal{Pl}})$. 
\end{remark}

\subsection{Lax--Phillips scattering theory} \label{sec_scattering_theory}

We now come to the analogue on the tree of the d'Alembert presentation of solutions of the wave equation, and the analogue of the expression in Proposition \ref{prop_energy_left_right} for the energy form. The results of this section are similar in some ways to those in the physics paper \cite{romanov_rudin_95}, though our focus is more on the geometric description of the translation representations, the scattering operator, and the expression of solutions as a superposition of plane waves. This section also bears some resemblance to \cite{colin_de_verdiere_truc}. However, overall, this section is most greatly inspired by and modeled on the paper of Lax--Phillips \cite{lax_phillips_h3}.

Suppose $H$ is a Hilbert space and $V$ is a unitary operator on $H$. Suppose $D_+ \subset V$ is a subspace such that
\begin{gather} \label{outgoing_space}
    \textnormal{(i) } V D_+ \subset D_+,  \hspace{5mm} \textnormal{(ii) } \bigcap_k V^k D_+ = \{0\}, \hspace{5mm} \textnormal{(iii) } \overline{\bigcup_k V^k D_+} = H.  
\end{gather}
In the language of Lax--Phillips \cite{lax_phillips_book}, the subspace $D_+$ is called \textit{outgoing}. We have the following:
\begin{theorem}[Theorem 1.1, Chapter II of \cite{lax_phillips_book}]
    If $D_+ \subset H$ is an outgoing subspace with respect to $V$, then there exists an isometry $W_+: H \to \ell^2(\mathbb{Z}; N)$, where $N$ is an auxiliary Hilbert space, such that $V$ maps to the right shift operator and $D_+$ maps to $\ell^2(\mathbb{Z}_{\geq 0}; N)$. This representation is unique up to an isomorphism of $N$.
\end{theorem}
\noindent The map in the theorem is called the \textit{outgoing translation representation}. We may similarly define the notion of an \textit{incoming subspace} $D_- \subset H$ which satisfies properties (ii) and (iii) of \eqref{outgoing_space}, but has property (i) replaced by $V^{-1} D_- \subset D_-$. We may subsequently similarly define the notion of an \textit{incoming translation representation} $W_-$, in which case $D_-$ maps to $\ell^2(\mathbb{Z}_{\leq -1}; N)$. Let $F: \ell^2(\mathbb{Z}; N) \to L^2(S^1; N)$ be the Fourier transform
\begin{gather*}
    [Fh](e^{i \theta}) = \sum_{n} h(n) e^{i n \theta}.
\end{gather*}
Then $F \circ W_+$ maps $V$ to the operator corresponding to multiplication by $e^{i \theta}$. We refer to this as the \textit{outgoing spectral representation}. 

\begin{remark} 
In a heuristic sense, in the above set up we should think of $V$ as a discrete wave propagator, and $N = L^2(X, \mu)$ for some space $X$ which parametrizes the directions that ``planes waves'' can travel in. Then the translation representation is closely related to expressing every wave as a superposition of ``plane waves.''
\end{remark}

Suppose the space $H$ with unitary operator $V$ admits both an outgoing and incoming subspace. Both the incoming and outgoing representations afford translation representations $W_{\pm}: H \to \ell^2(\mathbb{Z}; N)$. The operator $S = W_+ \circ W_-^{-1}$ on $\ell^2(\mathbb{Z}; N)$ is called the \textit{scattering operator}. Let $\mathcal{S} = F S F^{-1}: L^2(S^1; N) \to L^2(S^1; N)$. Then $\mathcal{S}(w)$, with $w \in \mathbb{C}$, may be viewed as a function on $S^1$ taking values in (unitary) operators $N \to N$. Often this extends to a meromorphic family of operators whose poles are called \textit{resonances}. 

\begin{remark} \label{remark_flat_scattering}
    Let $H$ denote the Hilbert space of initial data for the flat wave equation which have finite energy. Let $V = \mathcal{U}$ as defined in \eqref{eqn_one_time_step_prop_U}. Proposition \ref{prop_energy_flat} tells us that $\mathcal{U}$ acts unitarily on $H$. Let $D_+$ denote those initial conditions such that the associated solution to the wave equation $u(n, t)$ is supported outside of $[-t, t]$ for all $t \geq 0$. Then $D_+$ is exactly those initial conditions such that $u(n, t)$ can be written as $w_{\infty}(n-t) + w_{-\infty}(n+t)$ with $w_\infty(k)$ supported on $k \geq 1$ and $w_{-\infty}(k)$ supported on $k \leq -1$. As $t \to \infty$, we can think of $w_\infty$ as ``moving towards'' $\infty$ and $w_{-\infty}$ as ``moving towards'' $-\infty$. Let $N = \ell^2(\{\pm \infty\})$, i.e. we take the set with two elements which we label as $\pm \infty$ equipped with the counting measure; we denote elements here as $(\alpha, \beta)$ where $\alpha$ is the value at $\infty$, and $\beta$ is the value at $-\infty$. We define a map $W_+: H \to \ell^2(\mathbb{Z}; N)$ as follows. If $u(n, t) = w_\infty(n-t) + w_{-\infty}(n+t)$ is the solution to the wave equation corresponding to some choice of initial conditions, then we map these initial conditions to $t \mapsto (\frac{w_{\infty}(t+1) - w_{\infty}(t-1)}{2}, \frac{w_{-\infty}(-t+1) - w_{-\infty}(-t-1)}{2}) \in N$. Proposition \ref{prop_energy_left_right} tells us that this is an isometry. Elements in $D_+$ clearly map to $\ell^2(\mathbb{Z}_{\geq 0}; N)$, and $\mathcal{U}$ clearly maps to the right shift operator. We have thus found the outgoing translation representation. 
    
    Now let $D_-$ denote those initial conditions such that $u(n, t)$ is supported outside of $[-|t|-1, |t|+1]$ for $t \leq 0$. If we take $W_- = W_+$, then we in fact also obtain the incoming translation representation. In this case the scattering operator is just the identity. This makes sense as waves in the flat setting do not ``scatter''.
\end{remark}

Now consider the wave equation on the tree. We know that the matrix from \eqref{eqn_tree_one_time_step}, call it $\mathcal{V}$, acts unitarily on the space $\mathcal{H}$ of initial data. Let $D_+$ denote the subspace of initial data such that the corresponding solution to the wave equation $u(v, t)$ has the property that for all $t \geq 0$, we have that $u(v, t)$ is supported outside of the ball of radius $t$ centered at $o$. Clearly such solutions exist, for example the initial conditions corresponding to the Chebyshev polynomials $F_t$ (see \eqref{eqn_spherical_soln_tree}). It will turn out that $D_+$ is indeed outgoing (only property (iii) is not immediately clear). Furthermore, we may geometrically interpret the corresponding translation representation geometrically. The auxiliary Hilbert space in question will turn out to be $L^2(\Omega, \nu)$. In the sequel we shall often write $L^2(S^1 \times \Omega, d \theta \times d \nu)$, or even just $L^2(S^1 \times \Omega)$, in place of $L^2(S^1; L^2(\Omega, d \nu))$. The translation representation is very closely related to expressing general solutions to the wave equation as superpositions of ``plane waves'', i.e. waves which are constant on horocycles and which are ``purely outgoing''. 

Let us explain more precisely what we mean by horocyclic plane waves. Let $\omega \in \Omega$. Suppose $f: \mathbb{Z} \to \mathbb{C}$ is any function. Consider the functions on $\mathcal{T} \times \mathbb{Z}$ defined by
\begin{align*}
    F_+(v, t) &= q^{h_\omega(v)/2} f(h_\omega(v) - t) \\
    F_-(v, t) &= q^{h_\omega(v)/2} f(h_\omega(v) + t).
\end{align*}
It is immediate to show that both functions solve the wave equation. Furthermore, at each time $t$ they are clearly constant on horocycles of $\omega$. For $F_+$, as $t \to \infty$, the wave ``moves towards $\omega$'', and for $F_-$, as $t \to \infty$, the wave ``moves away from $\omega$''. 

The below theorem tells us that every solution to the wave equation is a superposition of such plane waves. Furthermore, similarly to the discussion in Remark \ref{remark_flat_scattering}, in order to go from the representation of each solution as a superposition of plane waves to the translation representation, we need to take an appropriate derivative in time. In Remark \ref{remark_flat_scattering}, that derivative spectrally corresponded to multiplying by $\frac{\xi - \xi^{-1}}{2}$. As appears in the below theorem, in the tree case the appropriate derivative spectrally corresponds to multiplying by $\frac{\xi - q^{-1} \xi^{-1}}{1 + q^{-1}}$. It is interesting to note that the ratio of the function in the flat vs tree case is (up to a constant) the $c$-function \eqref{eqn_c_function}. When we set $q = 1$, we get the infinite 2-regular tree, i.e. $\mathbb{Z}$, so we recover the flat case.

We now describe two closely related maps from $\mathcal{H}$ to $L^2(S^1 \times \Omega, d \theta \times d \nu)$. Let $f = (f_1, f_2) \in \mathcal{H}$. The first map $T_+$ is defined via
\begin{align*}
    (f_1, 0) &\mapsto [\textnormal{Hel}(f_1)](q^{-is}, \omega) \cdot c(q^{-is})^{-1} \cdot \Big( \frac{q^{is} - q^{-is}}{2} \Big) \\
    (0, f_2) & \mapsto -[\textnormal{Hel}(f_2)](q^{-is}, \omega) \cdot c(q^{-is})^{-1},
\end{align*}
and extended linearly. The second map $R_+$ is a bit more involved to describe. First let $u^f(v, t)$ denote the solution to the wave equation corresponding to the initial conditions. Define
\begin{gather}
    \hat{u}^f(v, q^{is}) := \sum_{t = -\infty}^{\infty} u^f(v, t) (q^{is})^t. \label{eqn_defn_u_f_hat}
\end{gather}
Since $u^f(v, t)$ satisfies the wave equation, we get that
\begin{gather*}
    \frac{A_v}{2 \sqrt{q}} \hat{u}^f(v, q^{is}) = \Big(\frac{q^{is} + q^{-is}}{2} \Big) \hat{u}^f(v, q^{is}).
\end{gather*}
We then define 
\begin{gather}
    [R_+ f](q^{is}, \omega) := \mathcal{P}^{-1}_{q^{is}} (\hat{u}^f(\cdot , q^{is})) (\omega) \label{eqn_inverse_poisson}
\end{gather}
Note that a priori it is not clear that $\hat{u}^f(v, q^{is})$ is well-defined (i.e. the series might not converge). Furthermore, in general the preimage of an eigenfunction under the Poisson transform is only a finitely additive measure, rather than a function on $\Omega$. That the above defined procedure is well-defined is part of the content of the below theorem.

Finally we define
\begin{align*}
    k_+^f (t, \omega) & := \frac{1}{2 \pi} \int_{S^1} [R_+ f](q^{is}, \omega)  \cdot (q^{is})^t  \ d \theta. \\
    g_+^f (t, \omega) & := \frac{1}{1 + q^{-1}} (k_+^f (t+1, \omega) - q^{-1} k_+^f (t-1, \omega)).
\end{align*}

\begin{theorem}
    Let $f \in \mathcal{H}$. Then
    \begin{enumerate}
        \item The map $R_+$ extends to a continuous map from $\mathcal{H}$ to $L^2(S^1 \times \Omega, d \theta \times d \nu)$. 
        \item The map $T_+$ is an isometry between $\mathcal{H}$ and $L^2(S^1 \times \Omega, d \theta \times d \nu)$.
        \item We have
        \begin{gather}
            u^f(v, t) = \int_{\Omega} q^{\frac{h_\omega(v)}{2}} k_+^f (h_\omega(v) - t, \omega) \ d \nu(\omega). \label{eqn_superposition_waves}
        \end{gather}
        \item We have the relation
        \begin{gather}
            T_+ = \frac{(q^{is} - q^{-1} q^{-is})}{(1 + q^{-1})} \cdot R_+. \label{eqn_relation_R_T}
        \end{gather}
         
    \end{enumerate}
\end{theorem}
\begin{proof}
    We begin by proving the Claim (2). For this we utilize the properties of the Helgason transform from Proposition \ref{prop_helgason_transform}. Suppose $f = (f_1, f_2)$. We know that $f_2 \in L^2(\mathcal{T})$. On the other hand, we know that $(f_1, 0)$ can be approximated in $\mathcal{H}$ by functions $(g, 0)$ with $g \in L^2(\mathcal{T})$. Therefore, on a dense subspace of $\mathcal{H}$, we have:
    \begin{align*}
        \langle (1 - \Big(\frac{A}{2 \sqrt{q}} \Big)^2) f_1, f_1 \rangle &= \int_\Omega \int_{S^1} (1 - (\frac{q^{is} + q^{-is}}{2})^2 ) \cdot [\textnormal{Hel}(f_1)](q^{is}, \omega) \cdot \overline{[\textnormal{Hel}(f_1)]}(q^{is}, \omega) |c(q^{is})|^{-2} d \theta d \nu \\
        &= \int_{\Omega} \int_{S^1} \Big( \frac{q^{is} - q^{-is}}{2} \Big)  c(q^{-is})^{-1} [\textnormal{Hel}(f_1)](q^{-is}, \omega)\\
        & \hspace{15mm} \cdot \Big(\frac{q^{-is} - q^{is}}{2} \Big) \overline{[\textnormal{Hel}(f_1)]}(q^{-is}, \omega) \overline{c(q^{-is})^{-1}} d \theta d \nu. \\
        &= \| T_+ (f_1, 0) \|^2 \\
        \langle f_2, f_2 \rangle &= \int_{\Omega} \int_{S^1} [\textnormal{Hel}(f_2)](q^{is}, \omega) \cdot \overline{[\textnormal{Hel}(f_2)]}(q^{is}, \omega) |c(q^{is})|^{-2} d \theta d \nu \\
        &= \int_{\Omega} \int_{S^1} [\textnormal{Hel}(f_2)](q^{-is}, \omega) \cdot \overline{[\textnormal{Hel}(f_2)]}(q^{-is}, \omega) |c(q^{-is})|^{-2} d \theta d \nu \\
        &= \| T_+(0, f_2) \|^2. 
    \end{align*}

    We now wish to show that $\langle T_+(f_1, 0), T_+(0, f_2) \rangle = 0$. For this it will suffice to consider the case when $f_1 = \delta_x$ and $f_2 = \delta_y$ for some points $x, y \in \mathcal{T}$. Notice that $[\textnormal{Hel}(\delta_x)](q^{is}, \omega) = q^{(\frac{1}{2} + is) h_\omega(x)}$. Therefore,
    \begin{gather*}
        \langle T_+(\delta_x, 0), T_+(0, \delta_y) \rangle = - \int_{S^1} \frac{q^{is} - q^{-is}}{2} |c(q^{-is})|^2 \int_\Omega q^{(\frac{1}{2} + is) h_\omega(y)} q^{(\frac{1}{2} - is) h_\omega(x)} d \nu d\theta.
    \end{gather*}
    Define
    \begin{gather*}
        G(q^{is}) := \int_\Omega q^{(\frac{1}{2} + is) h_\omega(y)} q^{(\frac{1}{2} - is) h_\omega(x)} d \nu.
    \end{gather*}
    Then the symmetry condition \eqref{eqn_helgason_symmetry} implies that $G(q^{-is}) = G(q^{is})$. We thus have
    \begin{gather*}
        \int_{S^1} \frac{q^{is} - q^{-is}}{2} |c(q^{is})|^2 G(q^{is}) d \theta = 0
    \end{gather*}
    because under $s \mapsto -s$, the integrand picks up a minus sign. We have thus shown that $T_+$ defines an isometry on a dense subspace of $\mathcal{H}$, and thus can be extended to all of $\mathcal{H}$.

    We now focus on proving Claim (4). For this we start by explicitly computing $R_+$ of $(\delta_o, 0)$ and $(0, \delta_o)$. In both cases, the fact that the initial data is radial implies that $\hat{u}^f(v, q^{is})$ is a radial function. On the other hand, it is an eigenfunction and hence a multiple of the spherical function. To compute which multiple, we simply need to compute the value at $o$. Using \eqref{eqn_inverse_spherical} we have that
    \begin{gather*}
        u^{(\delta_o, 0)}(o, t) = T_t\Big(\frac{A}{2 \sqrt{q}}\Big) \delta_o(o) = \int_{S^1} \frac{\xi^t + \xi^{-t}}{2} \frac{(1 + q^{-1})}{4 \pi} \frac{1}{|c(\xi)|^2} d \theta. 
    \end{gather*}
    Therefore, using \eqref{eqn_defn_u_f_hat} we get that $\hat{u}^{(\delta_o, 0)}(o, \xi) = \frac{1 + q^{-1}}{2} \frac{1}{|c(\xi)|^2}$. Analyzing the Chebyshev polynomials of the second kind, we get
    \begin{gather*}
        u^{(0, \delta_o)}(o, t) = U_{t-1}\Big(\frac{A}{2 \sqrt{q}} \Big) \delta_o (o) = \int_{S^1} \frac{\xi^t - \xi^{-t}}{\xi - \xi^{-1}} \frac{(1 + q^{-1})}{4 \pi} \frac{(\xi - \xi^{-1})(\xi^{-1} - \xi)}{(\xi - q^{-1} \xi^{-1})(\xi^{-1} - q^{-1} \xi)} d \theta.
    \end{gather*}
    Therefore $\hat{u}^{(0, \delta_o)}(o, \xi) = -(1 + q^{-1}) \frac{(\xi^{-1} - \xi)}{(\xi - q^{-1} \xi^{-1})(\xi^{-1} - q^{-1} \xi)}$. 

    Using Proposition \ref{prop_spherical_function_new_center}, we get that
    \begin{align*}
        (q^{is} - q^{-1} q^{-is}) [R_+(f_1, 0)](\omega) &= (1 + q^{-1}) c(q^{-is})^{-1} \frac{q^{is} - q^{-is}}{2} \sum_{v \in \mathcal{T}} q^{(\frac{1}{2} - is) h_\omega(v)} f_1(v) \\
        &= (1 + q^{-1}) c(q^{-is})^{-1} \frac{q^{is} - q^{-is}}{2} [\textnormal{Hel}(f_1)](q^{-is}, \omega) \\
        (q^{is} - q^{-1} q^{-is})[R_+(0, f_2)](\omega) &= -(1 + q^{-1}) c(q^{-is})^{-1} \sum_{v \in \mathcal{T}} q^{(\frac{1}{2} - is) h_\omega(v)} f_2(v) \\
        &= -(1 + q^{-1}) c(q^{-is})^{-1} [\textnormal{Hel}(f_2)](q^{-is}, \omega).
    \end{align*}
    Therefore, on a dense subspace of $\mathcal{H}$, the relation \eqref{eqn_relation_R_T} holds. However, as $\frac{1}{q^{is} - q^{-1} q^{-is}}$ is bounded on $S^1$, we have that the map $\frac{(1 + q^{-1})}{(q^{is} - q^{-1} q^{-is})} T_+$ is a continuous operator from $\mathcal{H}$ to $L^2(S^1 \times \Omega)$ and agrees with $R_+$ on a dense subspace. Hence we can define $R_+$ as the unique extension satisfying \eqref{eqn_relation_R_T}. This shows Claim (1).

    Fixing $v \in \mathcal{T}$ and $t \in \mathbb{Z}$ we have that
    \begin{gather*}
        \int_\Omega q^{h_\omega(v)/2} k_+^f(h_\omega(v) - t, \omega) d \nu = \int_\Omega q^{h_\omega(v)/2} \frac{1}{2 \pi} \int_{S^1} [R_+ f](q^{is}, \omega) (q^{is})^{h_\omega(v) - t} d \theta d \nu.
    \end{gather*}
    The function $q^{h_\omega(v)/2}$ is bounded as a function of $\omega$. We also know that $R_+ f$ is in $L^2(S^1 \times \Omega)$. Since $S^1 \times \Omega$ is compact, we know that $L^2 \subseteq L^1$. Therefore, by the Fubini--Tonelli theorem, we can change the order of integration. We obtain
    \begin{gather*}
        \frac{1}{2 \pi} \int_{S^1} (q^{is})^{-t} \int_\Omega q^{(\frac{1}{2} + is) h_\omega(v)} [R_+ f] (q^{is}, \omega) d \nu d \theta = \frac{1}{2 \pi} \int_{S^1} (q^{is})^{-t} \hat{u}^f(v, q^{is}) d \theta = u^f(v, t)
    \end{gather*}
    for $f$ in a dense subset.

    For any fixed $v, t$, the map sending $f \in \mathcal{H}$ to $u^f(v, t)$ is clearly continuous. On the other hand the map sending $f$ to the RHS of \eqref{eqn_superposition_waves} is also clearly continuous. These two maps agree on a dense set, hence they must agree everywhere. This shows Claim (3).

\end{proof}

We shall ultimately see that $T_+$ is essentially the Fourier transform of the outgoing translation representation, and thus $f \mapsto g_+^f$ is the translation representation. However, one subtlety remains. It is clear from \eqref{eqn_superposition_waves}, that if $k_+^f$ is supported on $\mathbb{Z}_{\geq 1} \times \Omega$ (and thus $g_+^f$ is supported on $\mathbb{Z}_{\geq 0} \times \Omega$), then $f$ lies in $D_+$. However, the converse is not clear, i.e. a priori it is possible that there are $f \in D_+$ for which $g_+^f$ is not supported on $\mathbb{Z}_{\geq 0} \times \Omega$. The following theorem rules out this possibility.

Suppose $F$ is a distribution on $\Omega$. We define $\tilde{F}(v)$ to be the function on $\mathcal{T}$ defined as the average value of $F$ on the cylindrical set $\Omega(v)$ (see Section \ref{sec_harmonic_analysis_trees} for the definition of $\Omega(v)$). We thus obtain a function satisfying the properties
\begin{gather*}
    \sum_{w : d(o, w) = d(o, v) + 1} \tilde{F}(w) = \begin{cases}
        (q+1) \tilde{F}(v) & v = o \\
        q \tilde{F}(v) & v \neq o.
    \end{cases}
\end{gather*}

\begin{theorem} \label{thm_zero_soln}
    Let $k_+$ be a distribution on $\mathbb{Z} \times \Omega$. 
    \begin{enumerate}
        \item We have that
    \begin{gather*}
        u(v, t) := \int_{\Omega} q^{h_\omega(v)/2} k_+(h_\omega(v) - t, \omega) d \nu(\omega)
    \end{gather*}
    is identically zero if and only if $\tilde{k}_+(t, o) = 0$, and for every fixed $v \neq o \in \mathcal{T}$, we have
    \begin{gather}
        \tilde{k}_+(t, v) = \begin{cases}
            q^{-t/2} a(v) & \textnormal{$t$ even} \\
            q^{-t/2} b(v) & \textnormal{$t$ odd},
        \end{cases} \label{eqn_condition_zero_soln}
    \end{gather}
    for some $a, b: \mathcal{T} \to \mathbb{C}$.

    \item We have that $u(v, t) = 0$ whenever $d(o, v) \leq t$ if and only if for all $t \leq 0$ we have $\tilde{k}_+(t, o) = 0$ and for all $t \leq 0$ and $v \neq o$ we have that \eqref{eqn_condition_zero_soln} holds.

    \end{enumerate}
\end{theorem}

\begin{proof}
    Suppose $u(v, t)$ is identically zero. Note that
    \begin{gather*}
        u(o, t) = \int_{\Omega} k_+(-t, \omega) d\nu = 0 = \tilde{k}_+(o).
    \end{gather*}
    Let $v_1, \dots, v_{q+1}$ be the vertices at distance 1 from $o$. Note that for any $\omega \in \Omega(v_j)$, we have that $h_\omega(v_j) = 1$ and $h_\omega(v_k) = - 1$ for any $k \neq j$. Therefore
    \begin{gather*}
        u(v_j, t) = \Big(\sum_{j \neq k}  q^{-1/2} \tilde{k}_+(-1 - t, v_k) \frac{1}{q+1} \Big) + q^{1/2} \tilde{k}_+(1 - t, v_j) \frac{1}{q+1} = 0.
    \end{gather*}
    On the other hand, we know that $\sum_{k} \tilde{k}_+(\ell, v_k) = (q+1) \tilde{k}_+(\ell, o) = 0$. We therefore have
    \begin{gather*}
        q^{-1/2} (-\tilde{k}_+(-1-t, v_j)) + q^{1/2} \tilde{k}_+(1-t, v_j) = 0
    \end{gather*}
    i.e. $\tilde{k}_+(t - 2, v_j) = q \tilde{k}_+(t, v_j)$ for every $t$. This implies that $\tilde{k}_+(\ell, v_j) = q^{-\ell/2} a$ for all $\ell$ even, and $\tilde{k}_+(\ell, v_j) = q^{-\ell/2} b$ for all $\ell$ odd, for some $a, b \in \mathbb{C}$.

    We now proceed inductively. Suppose we know that the claimed property holds for all $v$ such that $d(o, v) \leq n$. Fix some $v$ such that $d(o, v) = n$, and let $x$ be some vertex adjacent to $v$ such that $d(o, x) = n+1$. Let $o \to v_1 \to v_2 \to \dots \to v = v_n \to x = v_{n+1}$ be the path from $o$ to $x$. We have that
    \begin{gather*}
        \Omega = \Big(\bigsqcup_{j = 0}^{n} (\Omega(v_j) \setminus \Omega(v_{j+1})) \Big) \sqcup \Omega(x).
    \end{gather*}
    Note that on $\Omega(v_j) \setminus \Omega(v_{j+1})$, we have $h_\omega(x) = -(n+1) + 2j$. Therefore
    \begin{align*}
        u(x, t) &= q^{-(n+1)/2} \Big( \sum_{\substack{w \sim o \\ w \neq v_1}} \tilde{k}_+(-(n+1) - t, w) \frac{1}{q+1} \Big) \\
        &+ \sum_{j = 1}^n \Bigg( q^{(-(n+1)+2j)/2} \Big( \sum_{\substack{w \sim v_j \\ w \neq v_{j-1}, v_{j+1}}} \tilde{k}_+(-(n+1) + 2j - t, w) \frac{1}{q^{j}(q+1)} \Big) \Bigg) \\
        &+ q^{(n+1)/2} \tilde{k}_+((n+1) - t, x) \frac{1}{q^n (q+1)} = 0.
    \end{align*}
    We now use the inductive hypothesis to simplify each term and obtain a telescoping cancellation. First we remove the factor of $\frac{1}{q+1}$ from the expression. Then, for the first term, we know that
    \begin{gather*}
        \sum_{w \sim o} \tilde{k}_+(\ell, w) = (q+1) \tilde{k}_+(\ell, o) = 0
    \end{gather*}
    therefore the first term is
    \begin{gather*}
        q^{-(n+1)/2}(-\tilde{k}_+(-(n+1)-t, v_1) = -q^{t/2} \tilde{k}_+(0, v_1). 
    \end{gather*}
    We now analyze the $j$th term above. We know that
    \begin{gather*}
        \sum_{\substack{w \sim v_j \\ w \neq v_{j-1}, v_{j+1}}} \tilde{k}_+(\ell, w) + \tilde{k}_+(\ell, v_{j+1}) = q \tilde{k}_+(\ell, v_j)
    \end{gather*}
    Thus the $j$th term becomes
    \begin{align*}
        &q^{-(n+1)/2} (q \tilde{k}_+(-(n+1)+2j-t, v_j) - \tilde{k}_+((n+1) + 2j - t, v_{j+1}) \\
        = &q^{-t/2 + (j-1)} \tilde{k}_+(0, v_{j}) - q^{-t/2 + j} \tilde{k}_+(0, v_{j+1}),  
    \end{align*}
    unless $j = n$ in which case we get
    $q^{-t/2 + (n-1)} \tilde{k}_+(0, v_n) - q^{-(n-1)/2} \tilde{k}_+(n-1-t, x)$. 
    We thus obtain
    \begin{align*}
        &-q^{t/2} \tilde{k}_+(0, v_1) + q^{t/2} \tilde{k}_+(0, v_1) - q^{t/2 + 1} \tilde{k}_+(0, v_2) + q^{t/2 + 1} \tilde{k}_+(0, v_2) + \dots + q^{-t/2 + (n-1)} \tilde{k}_+(0, v_n) \\
        &- q^{-(n-1)/2} \tilde{k}_+(n-1-t, x) +q^{-(n+1)/2} \tilde{k}_+((n+1)-t, x) = 0.
    \end{align*}
    Therefore
    \begin{gather*}
        \tilde{k}_+(\ell-2, x) = q \tilde{k}_+(\ell, x).
    \end{gather*}

    If instead we only have that $u(v, t) = 0$ whenever $d(o, v) \leq t$, then the same inductive argument above applies.
\end{proof} 

Define $T^\vee_+: \mathcal{H} \to L^2(S^1 \times \Omega)$ via
\begin{gather*}
    [T^\vee_+ h](q^{is}, \omega) := [T_+ h](q^{-is}, \omega).
\end{gather*}

\begin{theorem}
    The space $D_+$ is outgoing. The map $T_+^\vee$ is the spectral translation representation, and thus the map $f \mapsto g_+^f(t, \omega)$ is the translation representation.
\end{theorem}

\begin{proof}
    We already know that $T_+$ is an isometry. The identity \eqref{eqn_superposition_waves} immediately implies that 
    \begin{gather*}
        k_+^{\mathcal{V} f}(t, \omega) = k_+^f(t - 1, \omega).
    \end{gather*}
    This combined with the relation \eqref{eqn_relation_R_T} immediately implies that $T_+^\vee$ maps $\mathcal{V}$ to the operator corresponding to multiplication by $q^{is}$. It is also clear via \eqref{eqn_superposition_waves} that the image of $D_+$ contains the subspace $L^2(\mathbb{Z}_{\geq 0} \times \Omega)$. We thus only need to show that in fact $D_+ = L^2(\mathbb{Z}_{\geq 0} \times \Omega)$.

    Theorem \ref{thm_zero_soln} tells us that any time an element in $D_+$ is represented by a distribution $k_+^f$ on $\mathbb{Z} \times \Omega$ as in \eqref{eqn_superposition_waves}, we have that for $t \leq 0$, 
    \begin{gather*}
        \langle k_+^f(t, \omega), 1_{\Omega(v)} \rangle = \begin{cases}
            q^{-t} a(v) & \textnormal{$t$ even}\\
            q^{-t} b(v) & \textnormal{$t$ odd}.
        \end{cases}
    \end{gather*}
    We know that for elements in $D_+$, the distribution $k_+^f$ is actually a function. We thus have that
    \begin{gather*}
        \langle k_+^f(t+1, v) - q^{-1}k_+^f(t - 1, v), 1_{\Omega(v)} \rangle = \langle g_+^f(t, v), 1_{\Omega(v)} \rangle = 0
    \end{gather*}
    for all $t \leq -1$, but the only function which pairs to zero against every $1_{\Omega(v)}$ is the zero function. Therefore, $g_+^f(t, \omega) = 0$ for $t \leq -1$. On the other hand, this function is exactly the image of $f$ under the Fourier transform of $T_+^\vee$. Therefore $D_+$ maps bijectively to $L^2(\mathbb{Z}_{\geq 0} \times \Omega)$. 
\end{proof}

Finally we discuss the intertwining operators. We can consider the space $D_-$ of incoming waves, namely those initial conditions such that the corresponding solution to the wave equation is supported outside of a radius $|t|+1$ centered at $o$ for $t \leq 0$. Let $R_-$ be the map defined in the same way as $R_+$ except in \eqref{eqn_inverse_poisson} we instead work with $\mathcal{P}^{-1}_{q^{-is}}$. We then define $T_-$ by the relation
\begin{gather*}
    T_- = \frac{(q^{-is} - q^{-1} q^{is})}{(1 + q^{-1})} R_-. 
\end{gather*}
We also define $T^\vee_-$ via $[T^\vee_- h](q^{is}, \omega) = [T_- h](q^{-is}, \omega)$. One can easily check by adapting the preceding arguments that $T_-^\vee$ is Fourier transform of the incoming translation representation. Furthermore, we clearly have
\begin{proposition}
    The Fourier transform of the scattering operator is given by
    \begin{gather*}
        [T_+^\vee \circ (T_-^\vee)^{-1} h](q^{is}, \omega) = \frac{(q^{-is} - q^{-1} q^{is})}{(q^{is} - q^{-1} q^{-is})} \mathcal{I}_{q^{is}}(h(q^{is}, \cdot))(\omega).
    \end{gather*}
\end{proposition}
\noindent Thus in particular we see the resonances, i.e. the poles of the scattering operator, are exactly at $\pm \frac{1}{\sqrt{q}}$.

\section{The wave equation on biregular trees} \label{sec_big_biregular_tree}

\subsection{Spherical harmonic analysis on biregular trees} \label{sec_harmonic_analysis_biregular}

We now let $\mathcal{T}$ denote the $(p+1, q+1)$-biregular tree. Let $G = \textnormal{Aut}(\mathcal{T})$. Let $e$ be some edge of $\mathcal{T}$, and let $o_p, o_q$ denote the $(p+1)$-regular and $(q+1)$-regular vertices of $e$, respectively. Note that $G$ acts transitively on edges, but not on vertices. Let $\mathcal{T}_q$ denote the vertices of $\mathcal{T}$ of degree $q+1$. Then $G/K_q$ may be identified with $\mathcal{T}_q$. Let $C(\mathcal{T}_q; \mathbb{C})$ denote the space of complex-valued functions on $\mathcal{T}_q$. Let $A^q_2$ denote the operator acting on this space via
\begin{gather*}
    [A^q_2](v) = \sum_{w : d(v, w) = 2} f(w),
\end{gather*}
for $v \in \mathcal{T}_q$. Let $H(G, K_q)$ denote the algebra generated by $A_2^q$. We can identify this with the algebra of bi-$K_q$-invariant functions on $G$ with convolution product. We say that a function in $C(\mathcal{T}_q; \mathbb{C})$ is radial if its value only depends on the distance from $o_q$. We have a natural identification between compactly supported radial functions and $H(G, K_q)$ (e.g. $A_2^q$ corresponds to the radial function supported on vertices at distance 2 from $o_q$). A natural vector space basis for $H(G, K_q)$ consists of radial functions which are supported on a given sphere of even radius centered at $o_q$. Let $A_0^q$ denote the identity operator on $C(\mathcal{T}_q; \mathbb{C})$, which in turn corresponds to the radial function which is 1 at $o_q$ and zero everywhere else.

The definitions given previously for the boundary $\Omega$, the Busemann functions $h_\omega$ (with $\omega \in \Omega$), and horocycles also applies to this setting, but with $o$ replaced by $o_q$. We also have a natural $K_q$-invariant measure on $\Omega$, which we also call $\nu$. Like before, $\nu$ is only quasi-invariant under $G$; its Radon--Nikodym derivative $\frac{d \nu(g^{-1} \omega)}{d \nu(\omega)}$ is given by $P(g.o_q, \omega)$ with
\begin{gather*}
    P(x, \omega) := (\sqrt{pq})^{h_\omega(x)}.
\end{gather*}

We also have a Satake isomorphism, $\textnormal{Sat}: H(G, K_q) \to \mathbb{C}[x, x^{-1}]^W$ via
\begin{gather*}
    [x^k](\textnormal{Sat}(f)) := (\sqrt{pq})^k \sum_{v \in h_\omega^{-1}(2k)} f(v).
\end{gather*}
This is in fact an algebra isomorphism. A straightforward computation shows that
\begin{gather}
    \frac{A_2^q}{2 \sqrt{pq}} - \frac{(p-1) A_0^q}{2 \sqrt{pq}} \mapsto \frac{x + x^{-1}}{2}. \label{eqn_B_q}
\end{gather}
Let $B_q$ denote the element on the LHS of \eqref{eqn_B_q}. 

We also have spherical functions and a spherical transform. For each $z \in \mathbb{C}$, there is a unique radial function $\phi_z: \mathcal{T}_q \to \mathbb{C}$ with $\phi_z(o_q) = 1$ and $B_q \phi_z = z \phi_z$. Since these functions are radial, they can be thought of as functions on non-negative even integers.

Let $a = \sqrt{pq}, b = \sqrt{\frac{q}{p}}$. Given $\xi \in \mathbb{C}^\times$, define
\begin{gather*}
    c(\xi) := \frac{(1 - a^{-1} \xi^{-1})(1 + b^{-1} \xi^{-1}) \xi}{\xi - \xi^{-1}}.
\end{gather*}
Let $b(\xi)$ denote the numerator of the $c$-function. Given $z \in \mathbb{C}$, write it as $z = \frac{\xi + \xi^{-1}}{2}$. Then explicitly, $\phi_{(\xi, \xi^{-1})} := \phi_z$ is given by
\begin{gather*}
    \phi_{(\xi, \xi^{-1})}(2 t) = \frac{1}{1 + q^{-1}} (\sqrt{pq})^{-t} (\xi^t c(\xi) + \xi^{-t} c(\xi^{-1})).
\end{gather*}
Let $H_t$ be the Chebyshev polynomials associated to $h(\xi)$ from Example \ref{ex_biregular}. Then, in terms of $z$ we have
\begin{gather*}
    \phi_z(2 t) = \frac{1}{1 + q^{-1}} (\sqrt{pq})^{-t} H_t(z).
\end{gather*}
See Appendix C.1 of \cite{parkison}.

As before each element $f \in H(G, K_q)$ can be written as $f = g(B_q)$ where $g$ is some polynomial. We then define the spherical transform $\textnormal{Sph}: H(G, K_q) \to \mathbb{C}[z]$ by sending $f$ to $g(z)$. This admits an extension to an isometry $L^2( G / \hspace{-1.5mm} / K_q) \to L^2(\mathbb{C}, d \mu_{\textnormal{Pl}})$ where $L^2( G / \hspace{-1.5mm} / K_q)$ are not necessarily compactly supported radial functions in $L^2(\mathcal{T}_q)$, and $d \mu_{\textnormal{Pl}}$ is the Plancherel measure. It is given explicitly by the measure described in Example \ref{ex_biregular} multiplied by the factor $\frac{(1 + q^{-1})}{4 \pi}$. As before we have an inversion formula: if $f \in H(G, K_q)$ and $v \in \mathcal{T}_q$ with $d(o, v) = 2 \ell$, then
\begin{gather}
    f(v) = \int_{S^1} \textnormal{Sat}(f)(\xi) \cdot \phi_{(\xi, \xi^{-1})}(\ell) d \mu_{\textnormal{Pl}}(\xi). \label{eqn_invert_satake_biregular}
\end{gather}

There also exists a Poisson transform and Helgason transform for biregular trees. However, we were not able to find a clear discussion of these in the literature. Upon modifying the proof of Theorem (1.2) of \cite{figa_talamanca_nebbia}, we obtain:

\begin{proposition}
    Let $f \in C(\mathcal{T}_q; \mathbb{C})$ be such that $B_q f = \frac{\xi + \xi^{-1}}{2} f$ with $\xi \neq \frac{1}{\sqrt{pq}}, -\sqrt{\frac{p}{q}}$. Write $\xi = (\sqrt{pq})^{i s}$. Then there exists a unique $m \in \mathcal{K}'(\Omega)$ such that
    \begin{gather*}
        f(v) = \mathcal{P}_{\xi}(m) := \int_{\Omega} (\sqrt{pq})^{(\frac{1}{2} + i s) h_\omega(v)} dm(\omega).
    \end{gather*}
\end{proposition}
\noindent We have intertwining operators $\mathcal{I}_\xi: \mathcal{K}(\Omega) \to \mathcal{K}(\Omega)$ defined in the analogous way. 

We also have a Helgason transform, initially defined only for $f \in C_c(\mathcal{T}_q; \mathbb{C})$ via
\begin{gather*}
    [\textnormal{Hel}(f)](\sqrt{pq}^{is}, \omega) := \sum_{x \in \mathcal{T}_q} f(x) (\sqrt{pq})^{(\frac{1}{2} + is) h_\omega(x)}.
\end{gather*}
The map extends to an isometric embedding of $L^2(\mathcal{T}_q)$ into $L^2(S^1 \times \Omega, d \mu_{\textnormal{Pl}} \times d \nu)$. Its image is those function $F(\xi, \omega)$ satisfying the same relation as in \eqref{eqn_helgason_symmetry}, but with $q$ replaced with $\sqrt{pq}$. 

\subsection{Wave equation on biregular trees} \label{sec_wave_eqn_biregular}
We say that $u(v, t): \mathcal{T}_q \times \mathbb{Z} \to \mathbb{C}$ satisfies the wave equation if
\begin{gather}
    B_q u(v, t) = \frac{u(v, t+1) + u(v, t-1)}{2}. \label{eqn_wave_biregular_tree}
\end{gather}
The analogous results to those for the regular tree can be proved with minimal changes to the arguments. We summarize these.

\begin{theorem}
    Let $h_1, h_2$ and $m_1, m_2$ be as in Theorem \ref{thm_solve_wave_tree}. Let $g_1, g_2: \mathcal{T}_q \to \mathbb{C}$. The unique solution $u(v, t)$ to \eqref{eqn_wave_biregular_tree} satisfying the $h_1, h_2$ initial conditions given by $g_1, g_2$, respectively (like in \eqref{eqn_tree_initial}) is given by \eqref{eqn_tree_soln}, but with $\frac{A}{2 \sqrt{q}}$ replaced by $B_q$.
\end{theorem}

\begin{proposition}
    Suppose $u(v, t)$ is a solution to the wave equation. Then
    \begin{gather*}
        E(t) := \sum_{v \in \mathcal{T}_q} \overline{u(v, t)} \cdot \Big(1 - B_q^2 \Big) u(v, t) + \sum_{v \in \mathcal{T}_q} \Big \| \frac{u(v, t+1) - u(v, t-1)}{2} \Big \|^2
    \end{gather*}
    is independent of $t$.
\end{proposition}

We now discuss the analogue of Proposition \ref{prop_reg_tree_chebs}, i.e. explicit formulas for certain special choices of Chebyshev polynomials. We shall ultimately be particularly interested in $R_t := \textnormal{Ch}_t^{x + b^{-1}}$.

\begin{proposition}
    For $t \geq 0$ we have that
    \begin{align}
        U_t(B_q) \delta_{o_q}(v) &= \begin{cases}
            \frac{(a^{-1})^{t - \ell + 1} - (-b^{-1})^{t - \ell + 1}}{a^{-1} + b^{-1}} a^{-\ell} & \textnormal{if $d(o_q, v) = 2 \ell \leq 2t$} \\
            0 & \textnormal{otherwise}
        \end{cases} \label{eqn_biregular_u_t}\\
        R_t(B_q) \delta_{o_q}(v) &= \begin{cases}
            a^{-t} & \textnormal{if $d(o_q, v) \leq 2 t$ and even} \\
            0 & \textnormal{otherwise}
        \end{cases}  \label{eqn_biregular_chebyshev}\\
        H_t(B_q) \delta_{o_q}(v) &= \begin{cases}
            a^{-t} & \textnormal{if $d(o_q, v) = 2t$} \\
            0 & \textnormal{otherwise}
        \end{cases} \notag
    \end{align}
\end{proposition}

\begin{proof}
    We begin by proving \eqref{eqn_biregular_u_t}. First we assume that $q > p$. We use \eqref{eqn_invert_satake_biregular}. The value of $U_t(B_q)$ at vertices of distance $2 \ell$ from $o_q$ is given by
    \begin{gather*}
        \frac{1 + q^{-1}}{4 \pi} \int_{S^1} \frac{\xi^{t+1} - \xi^{-t-1}}{\xi - \xi^{-1}} \frac{a^{-\ell}}{1 + q^{-1}} (\xi^\ell c(\xi) + \xi^{-\ell} c(\xi^{-1})) \frac{1}{c(\xi) c(\xi^{-1})} d \theta.
    \end{gather*}
    After pulling out the factor of $a^{-\ell}$, we must integrate over $S^1$ the following expressions:
    \begin{align*}
        - &\frac{\xi^{t + \ell + 2}}{(1 - a^{-1} \xi)(1 + b^{-1} \xi)}, \hspace{5mm} \frac{\xi^{t-\ell}}{(1 - a^{-1} \xi^{-1})(1 + b^{-1} \xi^{-1})} \\
        & \frac{\xi^{-t + \ell}}{(1 - a^{-1} \xi) (1 + b^{-1} \xi)}, \hspace{5mm}
        - \frac{\xi^{-t-\ell-2}}{(1 - a^{-1} \xi^{-1})(1 + b^{-1} \xi^{-1})}
    \end{align*}
    Since we are assuming that $q > p$, we have that $a, -b$ both lie outside the unit disk. Therefore, the first term does not contribute. Similarly, for the fourth term, we can expand the denominator as a power series in $\xi^{-1}$. If we then multiply by the numerator, we get a power series in $\xi^{-1}$ without constant term, so this term also does not contribute.

    The second and third terms will in fact contribute the same amount. We can compute it either using residues, or by simply reading off the constant term of the Laurent series expansion at $0$ (and multiplying by $2 \pi$). If $\ell < t$, then the constant term is zero. Otherwise we get for each term
    \begin{gather*}
        2 \pi \sum_{j = 0}^{t - \ell} (a^{-1})^j (-b^{-1})^{t - \ell - j} = 2 \pi \frac{(a^{-1})^{t - \ell + 1} - (-b^{-1})^{t - \ell + 1}}{a^{-1} + b^{-1}}.
    \end{gather*}
    Thus in the end we get exactly the claimed formula.

    We claim that \eqref{eqn_biregular_u_t} also holds when $q \leq p$. We can see this spectrally using the added Dirac delta term and appropriately modifying the above argument. However, we can also see this more directly as follows: if we treat $\sqrt{p}$ and $\sqrt{q}$ formally as variables, then the value of $U_t(B_q)$ on the sphere of radius $2 \ell$ satisfies a linear recurrence relation in these variables as $\ell$ varies. One can check that \eqref{eqn_biregular_u_t} indeed satisfies the recurrence, regardless of the relative size of $p$ and $q$. 

    The other two formulas then follow immediately by direct computation using \eqref{eqn_biregular_u_t}. 
\end{proof}

We also have the analogous scattering theory.
\begin{theorem}
    Let $D_+$ denote the subspace of initial conditions of the wave equation on $\mathcal{T}_q$ whose solution are supported outside a ball of radius $t$ centered at $o_q$ at time $t \geq 0$. Let $D_-$ be the subspace of initial conditions which are supported outside the ball of radius $|t|+1$ at time $t \leq 0$. Let $(f_1, f_2)$ be finite energy initial data for the wave equation. Define
    \begin{align*}
        T_{\pm} (f_1, f_2) &:= [\textnormal{Hel}](f_1)(\xi^{\mp 1}, \omega) c(\xi^{\mp 1})^{-1} \Big(\frac{\xi^{\pm 1} - \xi^{\mp 1}}{2} \Big) \mp [\textnormal{Hel}](f_2)(\xi^{\mp 1}, \omega) c(\xi^{\mp 1})^{-1} \\
        g_{\pm}^{(f_1, f_2)}(t, \omega) &:= \frac{1}{2 \pi} \int_{S^1} [T_{\pm}(f_1, f_2)](\xi, \omega) \cdot \xi^{\pm t} d \theta(\xi) \\
        R_{\pm} &:= \frac{b(\xi^{\pm 1})}{1 + q^{-1}} T_\pm \\
        k^{(f_1, f_2)}_{\pm} &:= \frac{1}{2\pi} \int_{S^1} [R_{\pm}(f_1, f_2)](\xi, \omega) \cdot \xi^{\pm t} d \theta(\xi).
    \end{align*}
    Then
    \begin{enumerate}
        \item The space $D_+$ is outgoing, and $D_-$ is incoming.
        \item The map $(f_1, f_2) \mapsto g^{(f_1, f_2)}_{\pm}$ is the outgoing/incoming translation representation.
        \item We have
        \begin{gather*}
            u^{(f_1, f_2)}(v, t) = \int_\Omega \sqrt{pq}^{h_\omega(v)} k_{\pm}^{(f_1, f_2)}(\frac{h_\omega(v)}{2} \mp t, \omega) d \nu(\omega).
        \end{gather*}
        \item The Fourier transform of the scattering operator is given by
        \begin{gather*}
            [\mathcal{S} h](\xi, \omega) = \frac{d(\xi^{-1})}{b(\xi)} \mathcal{I}_{\xi}(b(\xi, \cdot))(\omega).
        \end{gather*}
        In particular the resonances are at $\xi = \frac{1}{\sqrt{pq}}, - \sqrt{\frac{p}{q}}$.
    \end{enumerate}
\end{theorem}
\noindent If we set $p = 1$, we immediately recover the results of Section \ref{sec_scattering_theory}.

\subsection{Delocalization of eigenfunctions on biregular graphs}
As an application of the ideas of this paper, we extend a result of Brooks--Lindenstrauss \cite{brooks_lindenstrauss_graphs} about regular graphs to the setting of biregular graphs. Before stating the result, we first make a few remarks about the spectral theory of biregular graphs.

Suppose $\mathcal{G}$ is a finite $(p+1, q+1)$-regular graph. Let $V$ denote the set of vertices. Let $V_p$ and $V_q$ denote the vertices of degree $p+1$ and $q+1$, respectively. Let $A_2^p: L^2(V_p) \to L^2(V_p)$ be the operator corresponding to summing up over a sphere of radius 2 centered at a vertex of degree $p$. Let $A_2^q$ be defined analogously. Let $A$ be the adjacency operator on $\mathcal{G}$. Then $A^2 = (A_2^p + (p+1) I_p) + (A_2^q + (q+1) I_q)$, where $I_p$ and $I_q$ are the identity operators on $L^2(V_p)$ and $L^2(V_q)$. Suppose $\phi$ is a function on $V$. Let $\phi_p$ and $\phi_q$ denote the restriction to $V_p$ and $V_q$. If $\phi$ is an eigenfunction of $A$, then we immediately see that $\phi_p$ and $\phi_q$ are eigenfunctions of $A_2^p$ and $A_2^q$, respectively. Thus studying eigenfunctions of $A$ can be reduced to separately studying eigenfunctions of $A_2^p$ and $A_2^q$.

\begin{remark}
The main result of this section (Theorem \ref{thm_delocalization}) is a result about delocalization of eigenfunctions of $A_2^q$ in case $q > p$. As we now briefly explain, one might not expect the below result to hold in case $q < p$, at least not in full generality. We start by noting the relation $|V_p|(p+1) = |V_q|(q+1)$ as both sides count the total number of edges of $\mathcal{G}$. Suppose $q < p$. Then we necessarily have that $|V_p| < |V_q|$. In fact we have that $\frac{|V_q| - |V_p|}{|V_q|} = 1 - \frac{q+1}{p+1}$. We can write $A^q_2 + (q+1) I_q = B^T B$ where $B: L^2(V_q) \to L^2(V_p)$ is the restriction of $A$ to $L^2(V_q)$. Note that $B$ necessarily has a kernel of size $|V_q| - |V_p|$. Therefore $B_q = \frac{1}{2 \sqrt{pq}}(A_q^2 + (q+1) I_q) + \frac{-(p+q)}{2 \sqrt{pq}} I_q$ has at least the fraction $1 - \frac{q+1}{p+1}$ of its eigenvalues equal to $\frac{-(p+q)}{2 \sqrt{pq}}$. When we properly normalize the measures in Example \ref{ex_biregular} by multiplying by $\frac{1 + q^{-1}}{2 \pi}$, \eqref{eqn_biregular_dirac_delta} turns into $(1 - \frac{q+1}{p+1})\delta_{\frac{-(p+q)}{2 \sqrt{pq}}}(z)$. This explains the appearance of the Dirac delta in the Plancherel measure of the biregular tree from a more geometric perspective. Because of such a high multiplicity of this eigenspace, one might expect that it would be possible to find an eigenfunction in this space which contains most of its mass on a relatively small set, and thus the below theorem might not hold in this case, at least for eigenfunctions in this high multiplicity eigenspace. On the other hand, if $q > p$, then this problem does not arise. The assumption that $q > p$ will be used crucially in the proof.
\end{remark}

From now on we shall only deal with operators acting on $V_q$. For each $n \in \mathbb{N}$, we define $\tilde{S}_n: C(\mathcal{T}_q; \mathbb{C}) \to C(\mathcal{T}_q; \mathbb{C})$ via
\begin{gather*}
    [\tilde{S}_n f](v) = \frac{1}{\sqrt{pq}^n} \sum_{w: d(v, w) = 2n} f(w).
\end{gather*}
Notice that a sphere of radius $2 n$ is of size $(p q)^n \frac{q+1}{q}$, so we are essentially summing up over the sphere of radius $2 n$, and then dividing by the square root of the volume of this sphere. 

Let $S_n$ denote the natural descent of $\tilde{S}_n$ to $L^2(V_q)$. In the below theorem we have an assumption about the $L^r \to L^s$ norm for $S_n$ (with $r, s$ conjugate exponents). This is essentially an assumption about having few short cycles, but phrased in an operator theoretic way. For example, if we take $r = 1, s = \infty$, then we have that 
\begin{gather*}
    \| S_n \|_{L^1 \to L^\infty} \leq (pq)^{-n/2}
\end{gather*}
for all $n$ less than the injectivity radius of the graph.

\begin{theorem} \label{thm_delocalization}
    Let $\varepsilon > 0$. Let $r, s$ be such that $1 \leq r < 2$ and $2 < s \leq \infty$ with $\frac{1}{r} + \frac{1}{s} = 1$. Let $C > 0$ and $\alpha > 0$. Then there exists a constant $D > 0$ and an $N_0 \in \mathbb{N}$ (both depending on $\varepsilon, p, q, r, C, \alpha$) such that for all $N > N_0$ and all $(p+1, q+1)$-regular finite graphs $\mathcal{G}$ with $q > p$ satisfying
    \begin{gather}
        \| S_n \|_{L^r \to L^s} \leq C (pq)^{-\alpha n} \ \ \ \ \ \ \textnormal{for all $n \leq N$,} \label{eqn_hypothesis_operator_norm} 
    \end{gather}
    we have that for any $L^2$-normalized eigenfunction $\phi$ of $B_q$ (or, equivalently, of $A_2^q$) on $\mathcal{G}$, any subset $E \subset V_q$ satisfying
    \begin{gather*}
        \sum_{x \in E} |\phi(x)|^2 > \varepsilon
    \end{gather*}
    must be of size
    \begin{gather}
        |E| \geq D (p q)^{\delta N} \label{eqn_delta}
    \end{gather}
    where $\delta = \delta (\varepsilon, \alpha, r)$. 
\end{theorem}

\begin{remark}
    If we let $p = 1$, then we immediately recover the original result of \cite{brooks_lindenstrauss_graphs} for regular graphs. To the best of our knowledge no results of this form have previously been given for biregular graphs. In \cite{le_masson_sabri}, the results of the authors imply that the support $\Lambda$ of an eigenfunction on a biregular graph is necessarily of size $|\Lambda| \geq \frac{(pq)^{\ell_{\mathcal{G}}/4}}{4 q}$ where $\ell_{\mathcal{G}}$ is the largest $\ell$ such that every vertex of $\mathcal{G}$ contains at most one cycle of length at most $\ell$. Notice that the hypothesis \eqref{eqn_hypothesis_operator_norm} is satisfied in this case with $r = 1$ and $N = \ell_\mathcal{G}$. However, our result is significantly stronger in the sense that it holds for general subsets $E$, not only for the support.
\end{remark}

\begin{remark}
    In \cite{ganguly_srivastava}, the authors improve on the techniques of \cite{brooks_lindenstrauss_graphs} to improve the specific nature of the constant $\delta$ appearing in \eqref{eqn_delta}. It's likely that, by incorporating their techniques into our proof of Theorem \ref{thm_delocalization}, we can obtain a better value for $\delta$. However, we do not pursue optimizing constants in Theorem \ref{thm_delocalization} in this paper.
\end{remark}

Note that $B_q$ is a self-adjoint operator, and therefore its spectrum is real. We refer to any eigenvalue of $B_q$ in $[-1, 1]$ as \textit{tempered}; these are exactly those $\lambda$ expressible as $\frac{x + x^{-1}}{2}$ with $x \in S^1$. Any (necessarily real) eigenvalue outside of this range is called \textit{untempered}. These are the $\lambda$ that can be written as $\frac{x + x^{-1}}{2}$ with $x \in (1, \infty) \cup (-\infty, -1)$. 

As in Section \ref{sec_harmonic_analysis_biregular}, we let $a = \sqrt{pq}$ and $b = \sqrt{\frac{q}{p}}$. Since we are assuming $q > p$, we have that $b^{-1} < 1$. Let $R_t := \textnormal{Ch}_t^{x + b^{-1}}$ as before. The proof of Theorem \ref{thm_delocalization} uses two main ingredients. The first is the bound (or really explicit formula) for the kernel function of $R_t (B_q)$ in \eqref{eqn_biregular_chebyshev}. This formula immediately gives the following.
\begin{proposition} \label{prop_op_norm_bound}
    If $\mathcal{G}$ satisfies the hypothesis \eqref{eqn_hypothesis_operator_norm} of Theorem \ref{thm_delocalization}, then
    \begin{gather*}
        \|R_t(B_q)\|_{L^r \to L^s} \lesssim (pq)^{-\alpha t}
    \end{gather*}
    for all positive integers $t \leq N$.
\end{proposition}
\begin{proof}
    The proof is the exact same as the proof of Corollary 1 in \cite{brooks_lindenstrauss_graphs}. 
\end{proof}

The second ingredient is the construction of a specific operator which simultaneously has desirable geometric properties (in the sense of strong operator norm bounds) and spectral properties (having a very large eigenvalue for a given eigenfunction $\phi$, and not too negative of an eigenvalue for all other eigenfunctions).

\begin{lemma} \label{lemma_BL_kernel}
    Let $\varepsilon > 0$ be fixed. There exists an $N_0$ (depending on $\varepsilon, p, q,$ etc.) such that for any $\lambda \in \mathbb{R}$ and any $N > N_0$ we can find a polynomial $K_\lambda^N \in \mathbb{R}[z]$ satisfying
    \begin{enumerate}
        \item $K_\lambda^N$ is of degree at most $N$.
        \item There exist constants $B > 0$ and $\eta > 0$ depending on previously defined constants ($\varepsilon, \alpha, p, q$, etc.), but not on $N$, such that
        \begin{gather*}
            \| K_\lambda^N(B_q) \|_{L^r \to L^s} \leq B (pq)^{-\eta N}
        \end{gather*}
        for all graphs $\mathcal{G}$ satisfying the hypothesis \eqref{eqn_hypothesis_operator_norm}.
        \item We have that $K_\lambda^N (z) \geq -1$ for $z \in \mathbb{R}$, and $K_\lambda^N (\lambda) \geq \varepsilon^{-1}$.
    \end{enumerate}
\end{lemma}

We follow the same technique as in Lemma 2 of \cite{brooks_lindenstrauss_graphs} (or equivalently, Lemma 3.1 of \cite{brooks_lindenstrauss}). Before discussing the details, we first describe the strategy, and what modifications must be made in comparison to \cite{brooks_lindenstrauss_graphs}. Ultimately $K_\lambda^N$ will be an appropriate modification of the Fejer kernel:
\begin{align*}
    F_{M}(x) :&= \frac{1}{M}(x^{M/2} + x^{M/2-1} + \dots + x^{1-M/2} + x^{-M/2})^2 \\
    &= \frac{1}{M}(x^M + 2 x^{M-1} + \dots + (M-1) x + M + (M-1) x^{-1} + \dots + 2 x^{1-M} + x^{-M}),
\end{align*}
with $M \approx \varepsilon^{-1}$. Notice that $F_M(1) = M $. Furthermore the above kernel is positive for $x \in S^1$ (since it is expressible as a square) or on the real line (assuming $M$ is even). Given $\lambda = \frac{\xi_\lambda + \xi_{\lambda}^{-1}}{2}$, we shall seek a $d$ of comparable size to $N$ such that $(\xi_\lambda)^{d}$ is close to 1. This will in particular imply that $F_M((\xi_\lambda)^{d})$ is of size comparable to $M$. One of the key properties of the original construction of Brooks--Lindenstrauss was that
\begin{gather*}
    F_M(x^{d}) - 1 = \sum_{k = 0}^{M-1} \frac{2 (k+1)}{M} \frac{x^{(M-k)d} + x^{-(M-k)d}}{2} = \sum_{k = 0}^{M-1} \frac{2 (k+1)}{M} T_{(M-k)d}(\frac{x + x^{-1}}{2}), 
\end{gather*}
i.e. it is a sum Chebyshev polynomials of the first kind, with bounded coefficients, and each of degree at least $d$ (which is of comparable size to $N$). Let $L_N^\lambda \in \mathbb{R}[z]$ be the degree $M d$ polynomial such that $L_N^\lambda (\frac{x + x^{-1}}{2}) = F_M(x^d) - 1$ (we shall ultimately also want that $M d \leq N$). Because in the regular graph case (with degree of regularity $(q+1)$) $T_t(\frac{A}{2 \sqrt{q}})$ has strong decay of its kernel function in $t$ (i.e. \eqref{eqn_cheb_1_tree}), we obtain a bound like $\|L_{\lambda}^N(\frac{A}{2 \sqrt{q}})\|_{L^r \to L^s} \lesssim q^{-\eta N}$.

In our setting, we seek a symmetric Laurent polynomial which can be written as a sum of $R_t(\frac{x + x^{-1}}{2})$ with bounded coefficients and with each $t \geq d$. On the other hand, it must be large when $x^{d}$ is close to one, and must not be too negative at arbitrary values of $x$ on $S^1$ or the real line. Towards this end, we shall make use of the following proposition. 

\begin{proposition} \label{prop_cheb_1_identity}
    For every $k \geq 3$ and $0 \leq \ell \leq k-3$ we have that
    \begin{gather*}
        x^k + x^{-k} = R_k - b^{-1} R_{k-1} + (-(1 - b^{-2})) \sum_{j = 0}^\ell (-b^{-1})^\ell R_{k - 2 - \ell} - (1 - b^{-1})(-b^{-1})^{\ell+1} U_{k - 3 - \ell},
    \end{gather*}
    with each Chebyshev polynomial evaluated at $\frac{x + x^{-1}}{2}$. 
\end{proposition}
\begin{proof}
    Proof is a straightforward calculation.
\end{proof}

Suppose we now start with 
\begin{align*}
    \tilde{F}_M(x^d) :&= F_{M}(x^d) - 1 - \frac{M-1}{M}(x^d + x^{-d}) \\
    &= \frac{1}{M}(x^{M d} + 2 x^{(M-1)d} + \dots + (M-2) x^{2d} + (M-2) x^{-2d} + \dots + x^{-Md}). 
\end{align*}
We can use Proposition \ref{prop_cheb_1_identity} to express $x^{c d} + x^{-c d}$ as a sum of $R_t$'s, each time with a remainder corresponding to some multiple (of size comparable to $b^{-d}$ or smaller) of $U_{d - 3}$. More specifically we get that
\begin{align}
    & \tilde{G}_M^d(x) := \tilde{F}_{M}(x^d) + \Big(\frac{(1 - b^{-2})}{M} \sum_{k = 1}^M k (-b^{-1})^{(M-k)d} \Big) U_{d - 3} \label{eqn_kernel_1}\\
    &= \frac{1}{M} \Big( \sum_{k = 1}^{M-1} k \big( R_{(M-k+1)d} - b^{-1} R_{(M-k+1)d - 1} - (1 - b^{-2}) \sum_{j = 0}^{(M-k) d - 3} (-b^{-1})^j R_{(M-k+1)d - 2 - j} \big) \Big), \label{eqn_kernel_2}
\end{align}
again with all Chebyshev polynomials evaluated at $\frac{x + x^{-1}}{2}$. As long as $d$ is sufficiently large, we see that $\tilde{G}_M^d(x)$ is bounded from below by $-4$ for $x \in S^1$. If $d$ is even, then it is positive for $x \in \mathbb{R}$. Note that \eqref{eqn_kernel_2} consists of a sum of polynomially many (in $M$ and $d$) of the $R_t$'s, with each coefficient uniformly bounded (in $M$ and $d$), and with each $t \gtrsim d$. Since the above expression is a symmetric Laurent polynomial, we can find a polynomial $\tilde{K}_M^d \in \mathbb{R}[z]$ such that $\tilde{K}_M^d(\frac{x + x^{-1}}{2}) = \tilde{G}_M^d(x)$. Ultimately our $K_\lambda^N$ will be (a multiple of) $\tilde{K}_M^d$ for an appropriate choice of $d, M$.

We now prove an ancillary lemma which will help us find the appropriate $M$ and $d$.
\begin{lemma} \label{lemma_find_M_d}
    Let $\theta_0 \in \mathbb{R}$ and $M \in \mathbb{N}$ be given. There exists a $\beta > 0$ (depending on $M$) such that the following holds: given a fixed $0 < \gamma < \beta$, we have that for all $N$ sufficiently large there exists a $d \in 2 \mathbb{N}$ such that
    \begin{enumerate}
        \item We have $M d \leq N$.
        \item We have $|d \theta_0 \mod 2 \pi| \leq \frac{\gamma}{M}$.
        \item For some $c > 0$ (depending on all constants other than $N$) we have that $d \geq c N$. 
    \end{enumerate}
\end{lemma}

\begin{proof}
    Let $R = \frac{1}{4} \lfloor \frac{N}{M} \rfloor$. By Dirichlet's approximation theorem we know that there exists an $e \in \mathbb{N}$ with $1 \leq e \leq R$ such that $|e \theta_0 \mod 2 \pi| \leq \frac{2 \pi}{R}$. We shall ultimately take as our $d$ an appropriate multiple of $e$. So long as we multiply by an $\ell$ such that $\frac{2 \pi}{R} \ell \leq \frac{\gamma}{M}$, then we will satisfy the second condition, i.e. we need $\ell \leq \frac{\gamma}{2 \pi M}{R}$. So long as $\gamma$ is sufficiently small (in a way depending on $M$), we have that $\frac{\gamma}{2 \pi M} \leq \frac{1}{2}$. Let $k$ be the smallest multiple of $e$ such that $\frac{\gamma}{4 \pi M}{R} \leq k e$. We let $d = 2 \cdot k e$. Notice if $k = 1$, then clearly $d = 2 k e \leq 2 R$, and if $k \geq 2$, then $e < \frac{1}{4} R$, in turn implying that $d = 2 k  e \leq 2 R$. Therefore $Md = 2 M k e \leq M \frac{1}{2} \frac{N}{M} \leq N$. Thus Property (1) is satisfied. Also clearly $2 k \leq \frac{\gamma}{2 \pi M}{R}$, so Property (2) is satisfied. Finally we have that $d = 2  k e \geq \frac{\gamma}{2 \pi M} R \geq \frac{\gamma}{2 \pi M} \frac{N}{M} \frac{1}{8}$ for $N$ sufficiently large, so taking $c = \frac{\gamma}{16 \pi M^2}$, we get Property (3). 
\end{proof}

We can now prove Lemma \ref{lemma_BL_kernel}.

\begin{proof}{Proof of Lemma \ref{lemma_BL_kernel}.}
    Let $M = \lceil 128 \varepsilon^{-1} \rceil$. (In the sequel we shall at times have statements that are only true if $M$ is sufficiently large, in which case we replace our $M$ by an appropriately bigger $M$. In the end this will have no meaningful effect as $M$ is only used in the scope of the proof, and what is ultimatley important is that $M \geq \lceil 128 \varepsilon^{-1} \rceil$). Let $F_M(x)$ be the standard Fejer kernel. There exists a $\gamma > 0$ such that if $|\theta| < \frac{\gamma}{M}$, then $F_M(e^{i \theta}) \geq \frac{M}{2} + 4$ (assuming $M$ sufficiently large). Suppose $\lambda$ is tempered. Then we can write $\lambda = \frac{e^{i \theta_0} + e^{-i \theta_0}}{2}$ for some $\theta_0 \in \mathbb{R}$. Let $N$ be large. Let $d$ be as in Lemma \ref{lemma_find_M_d} with respect to $\theta_0, M, \gamma, N$ as described above. Let $c$ be the corresponding constant defined as in Property (3) of \ref{lemma_find_M_d}. Let $K_{\lambda}^N := \frac{1}{4} \tilde{K}_{M}^d$. Property (1) of Lemma \ref{lemma_find_M_d} guarantees that the degree of $K_{\lambda}^N$ is at most $N$, so Property (1) of the statement of Lemma \ref{lemma_BL_kernel} is satisfied. 

    We now further analyze \eqref{eqn_kernel_2}. By just using the triangle inequality, and the fact that $b^{-1} < 1$, and Proposition \ref{prop_op_norm_bound}, we get that
    \begin{gather*}
        \| K_{\lambda}^N(B_q) \|_{L^r \to L^s} \leq \sum_{k = d-2}^{M d} \| R_{k}(B_q) \|_{L^r \to L^s} \lesssim_{p, q} (pq)^{-\alpha d} \leq (pq)^{-c \alpha N}.
    \end{gather*}
    In the last step we have used Property (3) from Lemma \ref{lemma_find_M_d}. Thus Property (2) of Lemma \ref{lemma_BL_kernel} is verified.

    Finally we prove Property (3) of Lemma \ref{lemma_BL_kernel}. For this we analyze \eqref{eqn_kernel_1}. We have that 
    \begin{gather*}
        \Big| (1 - b^{-2}){M} \sum_{k = 1}^M k (-b^{-1})^{(M-k)d} \Big| \leq (1 - b^{-2}) M b^{-d}. 
    \end{gather*}
    By making $d$ large (which we can accomplish by making $N$ large), we can make the above expression as small as we like. For $x \in S^1$, we have that $U_{d-3}(\frac{x + x^{-1}}{2})$ is bounded in size by $d$, and for $x$ real with $|x| > 1$ it's bounded by $d |x|^d$ (because the function is invariant under $x \to x^{-1}$, we can always assume that $|x| \geq 1$). On the other hand, $F_M(x^d)$ is positive for $x \in S^1$, and is of size at least $\frac{M}{8} |x|^{2 d}$ if $x$ is real with $|x| > 1$ and $d$ is even. The term $(x^d + x^{-d})$ is of size at most 2 for $x \in S^1$, and of size at most $2 |x|^d$ for $x$ real with $|x| > 1$.

    Suppose $x \in S^1$. Let $d$ be large enough that $(1 - b^{-2}) M b^{-d} d < 1$ (this will hold for all $d$ large enough depending on $M, p , q$). We therefore get that $K_{\lambda}^N(\frac{x + x^{-1}}{2}) \geq \frac{1}{4} (0 - 1 - 2 - 1) \geq -1$. Suppose $x$ is real with $|x| > 1$. Then $K_{\lambda}^N(\frac{x + x^{-1}}{2}) \geq \frac{1}{4} (\frac{M}{8} |x|^{2 d} - 1 - 2 |x|^d - d |x|^d) \geq \frac{1}{4}|x|^d (\frac{M}{8} - 1 - 2 - 1) \geq \frac{M}{64}$ if $M$ is sufficiently large. Finally we have that, by Property (2) of Lemma \ref{lemma_find_M_d}, $|d \theta_0 \mod 2 \pi| \leq \frac{\gamma}{M}$. We therefore have that $F_M(e^{i d \theta_0}) \geq \frac{M}{2} + 4$. Therefore $K_{\lambda}^N(\lambda) \geq \frac{1}{4} (\frac{M}{2} + 4 - 4) = \frac{M}{8} + \frac{1}{8} \geq \varepsilon^{-1}$. Thus Property (3) of Lemma \ref{lemma_BL_kernel} is satisfied.

    If instead $\lambda$ is untempered, i.e. $\lambda = \frac{x + x^{-1}}{2}$ for some real $x \neq \pm 1$ and $|x| > 1$, then we simply define $K_\lambda^N$ to be $K_{1}^N$. That this satisfies Properties (1) and (2) are immediate from the preceding discussion. Property (3) is also clear except possibly for the fact that $K_\lambda^N(\lambda) \geq \varepsilon^{-1}$. However, this follow from the previous analysis that $K_{1}^N(x) \geq \frac{1}{4} |x|^d (\frac{M}{8} - 4) \geq \frac{M}{64} \geq \varepsilon^{-1}$. 
\end{proof}

Using Lemma \ref{lemma_BL_kernel}, we may complete the proof of Theorem \ref{thm_delocalization} in the exact same way as in \cite{brooks_lindenstrauss_graphs}. As such, we omit the remaining steps and simply refer to the section ``Proof of Theorem 1'' in \cite{brooks_lindenstrauss_graphs}.

\printbibliography

\vspace{10mm}

\textsc{Sorbonne University \\
\indent Institut de Mathématiques de Jussieu-Paris Rive Gauche} \\
\indent {\it Email address:} \texttt{clhpeterson1870@gmail.com}

\end{document}